\title{\large
 {\textbf     {%PINNING AND THE LOCAL LIMIT OF 
 ON PINNED FIELDS, INTERLACEMENTS, AND \\
 RANDOM WALK ON $(\mathbb{Z}/N \mathbb{Z})^2$
 }}}

\date{}

\documentclass[12pt, leqno]{article}
\usepackage[english]{babel}
\usepackage[%a4paper,
left=2.7cm,right=2.7cm,top=2.7cm,bottom=3.7cm]{geometry}
\usepackage{amsmath}
\usepackage{amssymb}
\usepackage{amsfonts}
\usepackage{amsthm}
\usepackage{amsfonts, amsmath, amssymb, amsthm, mathrsfs}
\usepackage{graphicx}
\usepackage{color}
\usepackage[sans]{dsfont}

\usepackage{color}

\usepackage{comment}
\usepackage[hang,flushmargin]{footmisc}

\numberwithin{equation}{section}

\newtheorem{thm}{Theorem}[section]
\newtheorem{lem}[thm]{Lemma}
\newtheorem{proposition}[thm]{Proposition}

\newtheorem{corollary}[thm]{Corollary}
\newtheorem*{Tbis}{Theorem \ref{T:7.12}'}

\theoremstyle{remark}
\newtheorem{rmk}[thm]{Remark}

\newtheorem*{proof10}{Proof of Lemma \ref{L:10.16}
}
\theoremstyle{definition}

\addtocounter{section}{-1}

\begin{document}
%\layout

\maketitle

\begin{center}
\vspace{-3cm}
\vspace{1.3cm}
Pierre-Fran\c cois Rodriguez$^1$ 

\vspace{0.9cm}
Preliminary draft
\end{center}
\vspace{0.4cm}
\begin{abstract}
\centering
\begin{minipage}{0.9\textwidth}
%\vspace{0.5cm}
We define two families of Poissonian soups of bidirectional trajectories on~$\mathbb{Z}^2$, which can be seen to adequately describe the local picture of the trace left by a random walk on the two-dimensional torus $(\mathbb{Z}/N \mathbb{Z})^2$, started from the uniform distribution, run up to a time of order $(N\log N)^2$ and forced to avoid a fixed point. The local limit of the latter was recently established in~\cite{CPV15}. Our construction proceeds by considering, somewhat in the spirit of statistical mechanics, a sequence of ``finite volume'' approximations, consisting of random walks avoiding the origin and killed at spatial scale $N$, either using Dirichlet boundary conditions, or by means of a suitably adjusted mass. By tuning the intensity $u$ of such walks with $N$, the occupation field can be seen to have a nontrivial limit, corresponding to that of the actual random walk. Our construction thus yields a two-dimensional analogue of the random interlacements model introduced in~\cite{Sz10} in the transient case. It also links it to the pinned free field in $\mathbb{Z}^2$, by means of a (pinned) Ray-Knight type isomorphism theorem.

%We give a description of the local limit as $N \to \infty$ of the vacant set left by random walk on the two-dimensional torus $T_N$, run up to a time ... and pinned at the origin, in terms of a Poisson soup of killed trajectories on $\mathbb{Z}^2$, with suitably tuned intensity $u_N$ and killing parameter $\epsilon_N$, conditioned to stay away from $0$.

\end{minipage}
\end{abstract}

\thispagestyle{empty}

\vspace{4.8cm}

\begin{flushleft}

$^1$Department of Mathematics \hfill May 2017 \\
University of California, Los Angeles \\
520, Portola Plaza, MS 6172\\
Los Angeles, CA 90095 \\
\texttt{rodriguez@math.ucla.edu}
\end{flushleft}

\newpage
\mbox{}
\thispagestyle{empty}
\newpage

\section{Introduction}
 Consider a simple random walk on the discrete $d$-dimensional torus of length $N$, started from the uniform distribution and run up to a suitable time $t_N=t_N(d)$, and observe the set of uncovered vertices (the vacant set), as $N$ tends to infinity, in an attempt to understand
\textit{how the walk tends to cover the torus}, i.e. what its vacant set looks like as $t_N$ is made to vary. For $d\geq 3$ and $t_N = u N^d$, with $u > 0$, a time scale which turns out to define a certain Poissonian regime for the excursions of the walk, this question has been the subject of extensive research in recent years, see in particular \cite{BS08}, \cite{Wi08}, \cite{Sz10}, \cite{TW11}, \cite{CT16}, see also the surveys \cite{CT12}, \cite{DRS14}, and references therein with regards to the related problem of disconnecting a discrete cylinder (with large base). In particular, a very fruitful idea has been to describe the local limit of the walk as $N \to \infty$ by a Poisson soup of bi-infinite trajectories on $\mathbb{Z}^d$, the so-called \textit{random interlacements}, cf. \cite{Sz10}, \cite{Sz12a}, to then study their connectivity properties, and to couple them suitably to the random walk in order to show that its vacant set exhibits a phase transition in $u$, essentially from having a unique giant component (at small $u$) to consisting of very small connected components, at sufficiently large $u$, see \cite{TW11}, \cite{CT16}, for precise statements.

For $d=2$, the analogous question is ill-posed, because the random walk will typically sweep all of a small region $A$ in the torus upon visiting it, thus yielding a trivial local limit. To remedy this, one can ``renormalize'' by introducing a penalizing effect, for instance by forcing the walk to avoid a fixed point (call it $0$). For a certain choice of $t_N$, see \eqref{E:1.4} below - but see Remark~\ref{R:final}~- the corresponding limit has been recently computed in \cite{CPV15}, see also \cite{CP16} (and \eqref{E:1.5} below) and the limiting occupation field has been given an interpretation in terms of the trace of a soup of \textit{tilted} trajectories. 

Building on these results, the present work introduces a certain Poissonian description of the random walk on the torus, akin to the interlacements in higher dimension. Our approach is perhaps best explained by analogy with the (massless) Gaussian free field pinned at the origin. We will in fact show there are deep ties between the two objects, a feature already hinted at in \cite{Sz12d}, \cite{Sz12e}. 
Thus, consider a massless free field $\varphi^N_{\cdot}= (\varphi^N_{x})_{x\in \mathbb{Z}^2}$ with $0$ boundary condition outside of a box $B_N \subset \mathbb{Z}^2$ of radius $N$ around the origin. As $N \to \infty$, $\varphi_{\cdot}^N$ delocalizes by recurrence, but the limiting field of increments 
\begin{equation}
\label{E:0.1}
\varphi^p_{\cdot} = \lim_N (\varphi_{\cdot}^N-\varphi_0^N)
\end{equation}
obtained by subtracting $\varphi_0^N$ everywhere, is well-defined (the limit in \eqref{E:0.1} is in distribution, see \eqref{E:2.16} below). 

In a similar vein, one can hope to describe the local limit of the pinned walk by considering a Poisson cloud of random walk trajectories (playing the role of $\varphi_{\cdot}^N$) \textit{killed at spatial scale $N$}, see \eqref{E:1.kill}, \eqref{E:2.kill} below for precise definitions, and thinning it by removing those trajectories hitting $0$. The analogy has its limitations due to the presence of an additional parameter, the intensity $u$ of trajectories in the picture, but the main result of this work is that this construction can indeed be carried out (the hastened reader is referred to the discussion surrounding \eqref{E:1.7}). As will turn out, $u$ needs to be suitably tuned with $N$ in order to yield a non-trivial limit. Moreover, the Poisson cloud at scale $N$ can be seen to converge as $N \to \infty$ to the soup of \textit{tilted} trajectories of \cite{CPV15} alluded to above, and further links the limiting occupation time profile of the Poisson cloud to the pinned free field in \eqref{E:0.1} by means of a suitable isomorphism theorem \textit{\`a la} Ray-Knight, see \cite{MR06}, and also \cite{Sz12c}, \cite{Sz12b}, for findings of similar flavor in higher dimensions.

%The above Poissonian description (which at this point is a representation \textit{in law}) naturally seems to arise as a print of the actual random walk, and offers the possibility to investigate other time scales $t_N$ that seem to display a similar behavior, see Remark \ref{R:final} for more on this. We hope to come back to these points in future work.
 
 We now describe our results more precisely. Let $P^N$ denote the canonical law of the symmetric simple random walk (SRW) on the two-dimensional torus $\mathbb T_N =(\mathbb{Z}/N \mathbb{Z})^2$, with uniformly chosen starting point, and denote by $(X_n)_{n \geq 0}$ the associated canonical process. For $t \geq 0$, define the set 
\begin{equation}
\label{E:1.2}
U_t^N =  \mathbb{T}_N \setminus \{ X_0,X_1,\dots, X_{\lfloor t \rfloor}  \} 
\end{equation}
of uncovered sites at time $t$. By \cite{CPV15}, Theorem 2.6, one has, for any finite set $ A \subset \mathbb{Z}^2$ containing the origin, and all $\alpha > 0$,
\begin{equation}
\label{E:1.3}
\lim_{N \to \infty} P^N \big[ \pi_N(A) \subseteq U_{t_N^{(\alpha)}}^N \, \big| \, 0 \in U_{t_N^{(\alpha)}}^N \big] = e^{-\frac{\pi}{2}  \alpha \,  \text{cap}(A)},
\end{equation}
where $\pi_N: \mathbb{Z}^2 \to \mathbb T_N$ is the canonical projection, $\text{cap}(\cdot)$ stands for the two-dimensional capacity, cf. \eqref{E:2.12} below, and
\begin{equation}
\label{E:1.4}
t_N^{(\alpha)} = \frac{2\alpha}{\pi} N^2 \log^2N.
\end{equation}
We also set
\begin{equation}
\label{E:1.4.0}
t_N = t_N^{(1)}
\end{equation}
for later reference. In words, \eqref{E:1.3} asserts that the law of the random set $ \mathbb{T}_N \setminus U_{t_N^{(\alpha)}}^N$ under the conditional measure $P^N[\ \cdot \ |\,  0 \in U_{t_N^{(\alpha)}}^N]$ converges in distribution (in the sense of finite dimensional marginals) towards a probability measure $Q^{\alpha}$ on the space of configurations $ \widetilde{\Omega} = \{0,1 \}^{\mathbb{Z}^2}$ (endowed with its canonical $\sigma$-algebra $\widetilde{\mathcal{F}}$ and coordinate maps $\widetilde{Y}_{x}$, $x\in \mathbb{Z}^2$) characterized by
\begin{equation}
\label{E:1.5}
Q^{\alpha}(\widetilde{Y}_x=0, \, x \in A) \stackrel{\text{def.}}{=} e^{-\frac{\pi}{2} \, \alpha \,  \text{cap}(A)}, 
\end{equation}
 for any finite set $A$ containing the origin. Our main results, see Theorems \ref{T:LIMIT} and \ref{T:10.40} below, yield a constructive definition of the measure $Q^{\alpha}$ in terms of the occupation fields of certain families (indexed by $N$) of Poisson clouds of bidirectional trajectories on $\mathbb{Z}^2$ forced to avoid $0$ and killed \textit{at spatial scale $N$}, in the limit as $N \to \infty$ and with a suitably adjusted parameter $u=u_N$ governing the density of trajectories entering the picture. (Note that we think of killed trajectories as entering a cemetery point $x_* \notin \mathbb T_N$ upon being killed and remaining there from then on forever so that there is a natural notion of time-shift on the trajectories.) The killing at spatial scale $N$ means either of the following,
\begin{align}
&\text{- imposing Dirichlet boundary conditions (i.e. killing the walk) outside $B_N$, or} \label{E:1.kill}
\\[0.2em]
&\text{- using a mass $\epsilon_N$, the parameter of an independent exponential killing time} \label{E:2.kill} 
\end{align}
with $\epsilon_N \stackrel{N}{\rightarrow}0$ tuned appropriately, see Proposition \ref{P:10.4} below. For simplicity, we focus on \eqref{E:1.kill} for the remainder of this introduction. We thus define a family $\omega_{N,u}$, $N \geq 1$, $u\geq 0$, of Poisson random measures of 
bidirectional nearest-neighbor trajectories modulo time-shift, whose forward and backward part are killed after finitely many steps. The action of the governing intensity measure can be
informally summarized as follows:
first, defining the measure (on $\mathbb{Z}^2$)
\begin{equation}
\label{E:1.6}
\rho_{A}^{0,N}(x)= P_x[\widetilde{H}_A > T_{B_N}]\, P_x[H_0 > T_{B_N}] \,1_A(x), \quad x \in \mathbb{Z}^2,
\end{equation}
where $P_x$ is the canonical law of SRW on $\mathbb{Z}^2$ started at $x$, $\widetilde{H}_K$ is the hitting time of $K$, $H_K$ the entrance time of $K$, and $T_K=H_{K^c}$, cf. Section \ref{S:1}, one has, for any $A \subset B_N$,
\begin{enumerate}
\item the number $N_A$ of trajectories (modulo time-shift) in the cloud $\omega_{N,u}$ entering the set $A$ is a Poisson random variable with parameter $u\,\rho_{A}^{0,N}(\mathbb{Z}^2)$,
\item given $N_A$, these trajectories are independent and identically distributed, their entrance point $X_0$ in $A$ is sampled according to $\tilde\rho_{A}^{0,N} =\rho_{A}^{0,N}/\rho_{A}^{0,N}(\mathbb{Z}^2)$, and, setting time to be $0$ when the trajectory first enters in $A$, their backward part is distributed according to a simple random walk started at $X_0$, conditioned to exit $B_N$ before returning to $A$, and killed upon exiting $B_N$, and their forward part follows the law of a simple random walk started at $X_0$, conditioned to avoid $0$ until exiting $B_N$, and killed upon doing so.
\end{enumerate}
Then, defining $\mathcal{I}(\omega_{N,u})$ to be the set of vertices visited by at least one trajectory in the cloud $\omega_{N,u}$, the \textit{interlacement set at level $u$}, we find that 
\begin{equation}
\label{E:1.7}
\begin{split}
&\text{for any $\alpha > 0$ and any sequence $u_N=u_N(\alpha)$ satisfying}\\
&\text{$u_N(\alpha) \sim  \frac{2}{\pi} \alpha \log^2N$, as $N \to \infty$, the law of $ \mathcal{I}^{N,\alpha} \equiv \mathcal{I}(\omega_{N,u_N(\alpha)})$ on $\widetilde{\Omega}$}\\
&\text{converges in distribution towards the measure $Q^{\alpha}$ defined in \eqref{E:1.5}.
}
\end{split}
\end{equation}
 where $f \sim g$ means $\lim_N f(N)/g(N)=1$, the convergence in distribution is in the sense of finite-dimensional marginals, and with hopefully obvious wording, the law of  $\mathcal{I}^{N,\alpha}$ refers to the law of the occupation field $(1\{ x \in  \mathcal{I}^{N,\alpha} \})_{x\in \mathbb{Z}^2}$. In fact, \eqref{E:1.7} can be strengthened, cf. Theorem \ref{T:LIMIT} and Corollary \ref{C:LIMIT}, to the statement that the random measure $\omega_{N,u_N(\alpha)}$ converges in a suitable sense to the Poisson point process of tilted random walks introduced in \cite{CPV15}. An important observation, which serves as the starting point of the above construction, is a particular representation of the two-dimensional capacity as appearing in \eqref{E:1.5}, see Lemma \ref{L:4.cap} below (and also Proposition \ref{P:10.4} and Remark \ref{R:mass},1) in the massive case \eqref{E:2.kill}), which naturally makes the measure $\rho_{A,N}$ from \eqref{E:1.6} appear.

 We now describe the links to the pinned free field $\varphi^p$ on $\mathbb{Z}^2$ (cf. \eqref{E:0.1}, and also \eqref{E:2.16} below), which can be found in Theorems \ref{T:7.12}' and \ref{T:7.29}. We keep our focus on \eqref{E:1.kill}, and refer the reader to Remark \ref{R:mass}, 3) to see how to deduce the following results by taking suitable limits of massive models. We denote by $L_x(\omega_{N,u})$, $x \in \mathbb{Z}^2$, the local time profile associated to $\omega_{N,u}$, i.e. $L_x(\omega_{N,u})$ collects the total amount of time spent at $x$ by any of the trajectories in the support of $\omega_{N,u}$ (for definiteness, note that $\omega_{N,u}$ is constructed using continuous-time trajectories with unit jump rates). Our results then show, see Lemma \ref{L:6.20} and Theorem \ref{T:7.29}, that for any $\alpha > 0$ and $u_N(\alpha)$ as in \eqref{E:1.7},
 \begin{equation}
 \label{E:1.8}
 L_{\cdot}\big(\omega_{N,\frac2\pi u_N(\alpha)}\big) \stackrel{d}{\longrightarrow}L_{\cdot,\alpha}, \quad \text{ as } N \to \infty
 \end{equation}
(in the sense of finite-dimensional marginals) and the limiting occupation field $(L_{x,\alpha})_{x \in \mathbb{Z}^2}$ satisfies a ``pinned isomorphism theorem''
  \begin{equation}
   \label{E:1.9}
  \Big(L_{x,\alpha} + \frac12 (\varphi^{\,p}_x)^{\,2}\Big)_{x\in \mathbb{Z}^2} \stackrel{\text{law}}{=} \Big( \frac12 \big(  \varphi^{\,p}_x + \sqrt{2\alpha} a(x) \big)^2 \Big)_{x\in \mathbb{Z}^2},
  \end{equation}
  where, on the left-hand side, $\varphi_{\cdot}^p$ is sampled independently from $L_{\cdot,\alpha}$ and $a(x)$, $x \in \mathbb{Z}^2$ is the potential kernel of simple random walk on $\mathbb{Z}^2$, see \eqref{E:2.1} below. Note that the pinning produces a spatially inhomogenous shift modulated by $a(\cdot)$ on the right-hand side of \eqref{E:1.9}. 
  
 The isomorphism \eqref{E:1.9} comes about from a corresponding statement in finite volume, interesting in its own right, see Theorems \ref{T:7.12} and \ref{T:7.12}'. Due to the ``hard'' killing constraint in \eqref{E:1.kill}, the measure $\omega_{N,u}$ can be naturally associated to the decomposition of a \textit{single} Markov chain on $B_N \cup \{ x_*\}$, conditioned to avoid $0$, into excursions from $x_*$ (up to a certain random time). This is reminiscent of certain approximation schemes for infinite volume quantities in higher dimensions, see \cite{Sz12c}, and also \cite{Lu16}. As it turns out, one can first apply the Ray-Knight theorem to the above (conditioned) Markov chain, and then disintegrate the Gaussian fields at the right values. The claim \eqref{E:1.9} then follows by taking a suitable limit as $N \to \infty$: if one chooses to keep $u$ fixed, the resulting equality in law is trivial, but \eqref{E:1.7} and \eqref{E:1.8} suggest that this can be precluded by boosting $u$ in the right way, and \eqref{E:1.9} arises as a result of this.

\bigskip
We now describe the organization of this article. Section \ref{S:1} introduces some notation and collects a few useful facts. Sections \ref{S:2} and \ref{S:4} comprise the construction delineated above, leading to \eqref{E:1.7}, for Dirichlet b.c. as in \eqref{E:1.kill}. Section \ref{S:2} deals mostly with considerations in finite volume. The capacity formula of Lemma \ref{L:4.cap} naturally leads to a notion of interlacements avoiding a (finite) set $K$ and killed upon exiting a larger set $K'\supset K$, see in particular Theorem \ref{T:2}. In Section \ref{S:4}, we then perform the infinite-volume limit, with $K = \{ 0\}$, $K' =B_N$ and suitably tuned intensity $u=u_N$, to recover the (tilted) interlacements of \cite{CPV15}. The main results are Theorem~\ref{T:LIMIT} and Corollary~\ref{C:LIMIT}. Section \ref{S:5} deals with the connections to $\varphi_{\cdot}^{\, p}$ and the corresponding pinned isomorphisms, in finite and inifinite volume, see Theorems \ref{T:7.12} and \ref{T:7.29}. Finally, Section \ref{S:mass} entails the approximation by means of massive models. After some preparatory work (in particular, to determine the right scaling of $\epsilon_N$), the main result comes in Theorem \ref{T:10.40}. We conclude by sketching how to recover results from Section \ref{S:5} using this approach, and with a few general remarks.

\section{Preliminaries}
\label{S:1}
We consider the lattice $\mathbb{Z}^2$, with its usual nearest-neighbor graph structure, and the continuous-time symmetric simple random walk on $\mathbb{Z}^2$ with exponential holding times of parameter $1$. The canonical law of the walk started at $x$ is denoted by $P_x$, the corresponding expectation by $E_x$ and the canonical coordinates by $X_t$, $t \geq 0$. We write $Z_n$, $n \geq 0$, for the discrete skeleton of this walk, so that $X_t=Z_{N_t}$, $t\geq0$, where, under $P_x$, $(N_t)_{t\geq 0}$ is a Poisson process of rate $1$, independent of $(Z_n)_{n \geq 0}$. 
We introduce the potential kernel $a(\cdot)$ of the walk, defined as
\begin{equation}
\label{E:2.1}
a(x)= \lim_{n \to \infty}\sum_{k=0}^n\big( p_k(0)- p_k(x) \big), \text{ for } x \in \mathbb{Z}^2
\end{equation}
where 
\begin{equation}
\label{E:2.1.p}
p_k(x)=P_0[Z_k=x], \quad k \geq 0, x \in \mathbb{Z}^2,
\end{equation}
(see for instance \cite{La91}, Thm. 1.6.1 for well-definedness). One knows, see e.g. the proof of Thm. 1.6.1 in \cite{La91}, that the convergence in \eqref{E:2.1} is absolute,
\begin{equation}
\label{E:2.1.2}
\sum_{k=0}^{\infty} \big| p_k(0)- p_k(x) \big| < \infty.
\end{equation}
The function $a(\cdot)$ is non-negative, symmetric, i.e. $a(x)=a(-x)$, $x \in \mathbb{Z}^2$, $a(0)=0$, and 
\begin{equation}
\label{E:2.1.1}
\frac14 \sum_{y:\, y\sim x}( a(y) -a(x) )= \delta_0(x), \quad \text{ for all $x \in \mathbb{Z}^2$}.
\end{equation}
In particular, it is harmonic in $\mathbb{Z}^2 \setminus \{0\}$. Moreover, (cf. \cite{Sp76}, p.123, Prop. P2), 
\begin{equation}
\label{E:2.2}
\lim_{|x| \to\infty} a(x+x')-a(x) = 0, \text{ for all }x \in \mathbb{Z}^2,
\end{equation}
one has the asymptotics (see \cite{La91}, Thm. 1.6.2 and p.39)
\begin{equation}
\label{E:2.3.0}
a(x)=  \frac{2}{\pi} \log|x| + k + O(|x|^{-2}), \text{ as $x \to \infty$},
\end{equation}
for some positive constant $k$, and in particular, 
\begin{equation}
\label{E:2.3}
a(x) \sim \frac{2}{\pi}\log|x|, \text{ as $x \to \infty$}.
\end{equation}

Next, we collect a (gradient) estimate for the heat kernel of the continuous-time walk $(X_t)_{t\geq 0}$, which will be useful below. Let 
\begin{equation}
\label{E:2.3.5}
q_t(x)=P_0[X_t =x], \quad \text{for $x\in \mathbb{Z}^2$, $t \geq 0$.}
\end{equation}
As a consequence of the local central limit theorem, one has the bounds (see for instance \cite{LL10}, Thm. 2.3.6 - this result, stated for discrete-time walks, is easily transferable to the continuous time setting, cf. also the proof of (2.9), p.27 in \cite{LL10}),
\begin{equation}
\label{E:2.3.6}
|\nabla_x q_t(z) - \nabla_x \bar q_t(z)| \leq \frac{ c|x|}{t^2}, \quad \text{ for }x,z \in \mathbb{Z}^2, t \geq 0
\end{equation}
(with a constant $c$ independent of $x,z$ and $t$), where
\begin{equation}
\label{E:2.3.7}
\nabla_y f (x) = f(x+y)-f(x), \quad \text{for $f: \mathbb{Z}^2 \to \mathbb{R}$}, \,  x,y \in \mathbb{Z}^2,
\end{equation}
and
\begin{equation}
\label{E:2.3.8}
\bar q_t(x) = \frac{c}{ \pi t} e^{- \frac{c'|x|^2}{2t}}, \quad  x \in \mathbb{Z}^2, t\geq 0,
\end{equation}
for suitable $c,c' \in (0,\infty)$ (in fact $c=1$ and $c'=2$).

We recall some more elements of potential theory for the random walk. Given $K \subset \mathbb{Z}^2$, we write $H_K$, $\widetilde{H}_K$, and $T_K=H_{\mathbb{Z}^d\setminus K}$ for the entrance, hitting times of $K$ and exit times from $K$, respectively. If $K=\{x\}$ is a singleton, we simply write $H_x$ etc. For a finite set $A\subset \mathbb{Z}^2$, one defines the harmonic measure of $A$ (from infinity) by
\begin{equation}
\label{E:2.10}
\text{hm}_A(x)\stackrel{\text{def.}}{=} \lim_{|y|\to \infty}P_y[X_{\widetilde{H}_A} = x] , \quad x \in \mathbb{Z}^2
\end{equation}
(see for instance \cite{LL10}, Ch. 6.6 for the existence of this limit). Note that $\text{hm}_A$ is a probability measure on $A$, as follows by recurrence of the walk. It is well known that $\text{hm}_A(x)$ can be expressed in terms of escape probabilities from $A$, see \cite{LL10}, Prop. 6.6.1 and (6.44), as 
\begin{equation}
\label{E:2.11}
\text{hm}_A(x)=  \lim_{N \to \infty} \frac2\pi (\log N) P_x[T_{ B_N} < \widetilde{H}_A] , \quad x \in A.
\end{equation}
The two-dimensional capacity of a finite set $A \subset \mathbb{Z}^2$ is then defined as
\begin{equation}
\label{E:2.12}
\text{cap}(A) = \sum_{x\in A} a(x-y) \text{hm}_A(x), \quad \text{ for \textit{any} $y \in A$}.
\end{equation}
This is well-defined, i.e. does not depend on the choice of $y$, see \cite{LL10}, p.146. Moreover, if $A$ contains the origin, we will usually set $y=0$.  
%STOP HERE
We will also need some elements of potential theory for the killed walk.
We write $g_K$ for the Green function of the walk on $\mathbb{Z}^2$ killed outside $K$, for finite $K \subset \mathbb{Z}^2$:
\begin{equation}
\label{E:2.13}
g_K(x,y)= E_x \Big[ \int_0^{\infty} \text d t \, 1\{ X_t=y, t< T_K\}\Big], \text{ for }x,y \in \mathbb{Z}^2.
\end{equation}
Supposing that $K' \subset K$ and letting $U = K \setminus K'$, the Green functions $g_K$ and $g_{K'}$ are related via 
\begin{equation}
\label{E:2.13.0}
g_{K}(x,y)= g_{K'}(x,y) + E_x[H_U < T_K, \, g_K(X_{H_U},y)], \quad \text{for }x ,y \in \mathbb{Z}^2,
\end{equation}
which follows from an application of the strong Markov property (at time $H_U$). By Proposition 1.6.3 of \cite{La91}, the Green function for the walk killed outside of a set $K$ can be expressed in terms of the potential $a(\cdot)$ as
\begin{equation}
\label{E:4.0}
g_K(x,y) = E_{x}[a(X_{T_K} - y)] - a(x-y), \text{ for } x, y \in \overline{K}
\end{equation}
(note that in higher dimension, a similar formular holds true, with $a$ replaced by $g$, the infinite-volume Green function of the walk). The equilibrium measure of the set $A$ relative to $K$ for any $A \subset K \subset \subset \mathbb{Z}^d$ is defined as
\begin{equation}
\label{E:2.14}
e_{A,K}(x)= P_x[\widetilde{H}_K >T_{K}] 1_{A}(x), \text{ for } x \in \mathbb{Z}^2,
\end{equation}
along with the normalized equilibrium measure $\tilde e_{A,K}(x) = e_{A,K}(x)/ \sum_{y \in A} e_{A,K}(y)$, which is a probability measure on $A$ (supported on its interior (vertex) boundary). The measure $e_{A,K}$ satisfies the following sweeping identity: for all $A'\subset A'\subset K$, and $x \in \mathbb{Z}^2$, one has
\begin{equation}
\label{E:2.15}
e_{A',K}(x) = P_{e_{A,K}}[H_{A'}< T_B, X_{H_{A'}} = x],
\end{equation}
where we define $P_{\mu} = \sum_{x}\mu(x)P_x$, for any measure $\mu$ on $\mathbb{Z}^2$. 

We note the following bound for exponential moments of exit times.

%Suppose for the purpose of the following lemma that $P_0$ also supports an exponential random variable $ \xi(\epsilon)$ of parameter $\epsilon > 0$, independent of $X_{\cdot}$.
\begin{lem}$\quad$
\label{L:2.exp}

\medskip 
\noindent There exists $c \in (0,\infty)$ such that, for all $R \geq 1$, one can find $\epsilon_0(R) > 0$ with
\begin{equation}
\begin{split}
\label{E:2.exp1}
E_0[e^{-\epsilon T_{B_R}} ] \leq 1- c \epsilon R^2, \quad \text{for all } \epsilon \leq \epsilon_0(R).
\end{split}
\end{equation}
\end{lem}
\begin{proof}
Let $T=T_{B_R} \ (\geq 0)$, and $R \geq 1$. By a classical argument due to Khas'minskii, see for instance \cite{AS82}, Theorem 1.2., one writes, for $n \geq 1$,
\begin{equation}
\label{E:2.exp2}
\begin{split}
E_0[T^n ] &= n! E_0\Big[ \int_{0\leq t_1\leq \dots \leq t_n< \infty} \text d t_1\dots \text d t_n 1\{ T > t_n\} \Big]\\
&\leq n! \int_{0\leq t_1\leq \dots \leq t_{n-1}< \infty} \text d t_1\dots \text d t_{n-1} E_0\Big[ 1\{ T > t_{n-1}\} E_{X_{t_{n-1}}} [ T] \Big],
\end{split}
\end{equation}
where the second line follows from the simple Markov property (at time $t_{n-1}$). One knows, see for instance \cite{Sz12d}, above (1.22), that for suitable $c_1 \in (0,\infty)$,
\begin{equation}
\label{E:2.exp3}
\sup_{x\in B_R} E_x[T_{B_R}] \leq c_1R^2.
\end{equation}
On the event $\{ T> t_{n-1}\}$ appearing in \eqref{E:2.exp2}, one has $\{ X_{t_{n-1}} \in B_R\}$ and therefore, substituting \eqref{E:2.exp3} yields the estimate
\begin{equation}
\label{E:2.exp4}
E_0[T^n ] \leq  n! (c_1R^2)^n, \quad n \geq 0.
\end{equation}
The claim then follows by choosing, say, $\epsilon_0(R) = c_1R^2/2$, expanding the left-hand side of \eqref{E:2.exp1}, for $\epsilon \leq \epsilon_0(R)$, using \eqref{E:2.exp4} and summing the resulting geometric series, and noting that $(1+x)^{-1}\leq 1-x/2 $, for $x \in [0,1/2]$.
\end{proof}

We now introduce the Gaussian fields that will be relevant in the sequel. For $N \geq 1$, we define the probability measure $\mathbf{P}_N^G$ on $\overline{\Omega}= \mathbb{R}^{\mathbb{Z}^2}$, endowed with the product $\sigma$-algebra $\overline{\mathcal A}$ and canonical coordinates $\varphi_x: \overline{\Omega}\to \mathbb{R}$, $\overline{\omega} = (\overline{\omega}_y)_{y\in \mathbb{Z}^2} \mapsto \varphi_x(\overline{\omega}) = \overline{\omega}_x$ such that, under $\mathbf{P}_N^G$, the canonical field $\varphi=(\varphi_x)_{x\in \mathbb{Z}^2}$ is the centered Gaussian field with covariance
\begin{equation}
\label{E:7.9}
\mathbf{E}^G_N[\varphi_x\varphi_y]=g_{B_N}(x,y), \quad x,y \in \mathbb{Z}^2,
\end{equation}
with $g_{B_N}(\cdot,\cdot)$ denoting the Green function of simple random walk killed outside $B_N$, cf. \eqref{E:2.13}. In particular $\varphi_x= 0$ whenever $x \notin B_N$ under $\mathbf{P}_N^G$. We next recall the definition of the Gaussian free field $\varphi^{\,p}_{\cdot}$ pinned at the origin. Its law $\mathbf{P}^G$ on $\overline \Omega$ is such that
\begin{equation}
\label{E:2.16}
\begin{split}
&\text{under $\mathbf{P}^G$, $\varphi^{\,p}_x$, $x \in \mathbb{Z}^2$ is a centered Gaussian field with}\\
&\text{covariance function $\mathbf{E}^G[\varphi^{\,p}_x \varphi^{\,p}_y] = a(x)+ a(y)- a(y-x)$, for $x,y \in \mathbb{Z}^2$},
\end{split}
\end{equation}
with $a(\cdot)$ given by \eqref{E:2.1}. This covariance function is symmetric in $x$ and $y$ since $a(\cdot)$ is an even function, and $\varphi^{\,p}_0=0$ as $a(0)=0$. The relation between $\mathbf{P}_N^G$ and $\mathbf{P}^G$ is the following. Using \eqref{E:4.0} and \eqref{E:2.2}, one can show (cf. also \eqref{E:12.2} below for a similar calculation) that 
\begin{equation}
\label{E:2.17}
\mathbf{E}^G[\varphi^{\,p}_x \varphi^{\,p}_y] = \lim_{N \to \infty} \mathbf{E}_N^G[(\varphi_x-\varphi_0)(\varphi_y - \varphi_0)],
\end{equation}
which gives an interpretation of $\varphi^{\,p}_{\cdot}$ as the limiting field of increments of $\varphi_{\cdot}$  at the origin under $\mathbf{P}_N^G$, as $N \to \infty$.

We conclude this section with some considerations at positive mass. The following results won't be needed until Section \ref{S:mass}. We introduce, for a parameter $\epsilon > 0$,
\begin{equation}
\label{E:10.0}
\xi \equiv \xi(\epsilon) \text{ an exponential random variable with parameter $\epsilon$},
\end{equation}
and, whenever needed, tacitly enlarge our canonical space such that $P_x$, cf. above \eqref{E:2.1}, also carries $ \xi(\epsilon)$ as in \eqref{E:10.0}, independently of $X_{\cdot}$. We define the (Green) function
 \begin{equation}
\label{eq:GF0}
\begin{split}
g_{\epsilon}(x,y) &= E_{x}\Big[\int_0^{\xi(\epsilon)} \text d t \, 1\{ X_t =y\}\Big], \qquad x, y \in \mathbb{Z}^2, \, \epsilon > 0.
\end{split}
\end{equation}
For all $\epsilon > 0$ and all $x, y$, $g_{\epsilon}(x,y)$ is finite, and $g_{\epsilon}(\cdot,\cdot)$  is the kernel of a positive definite, symmetric operator (on $\ell^2(\mathbb{Z}^d)$). It can thus be viewed as the covariance function of a centered Gaussian field, whose law $\mathbf{P}^G_{\epsilon}$ on $\overline{\Omega}$, cf. above \eqref{E:7.9}, is such that
\begin{equation}
\label{E:massGFF}
\mathbf{E}^G_{\epsilon}[\varphi^{\epsilon}_x\varphi^{\epsilon}_y] = g_{\epsilon}(x,y), \quad \text{for all }x,y \in \mathbb{Z}^2
\end{equation}
(with $\varphi_x^{\epsilon}$ denoting canonical coordinates on $\overline{\Omega}$, as before, the superscript is merely to avoid confusion).

\begin{lem}$(x,y \in \mathbb{Z}^2, \epsilon > 0)$
\label{L:10.20}
\begin{align}
&\lim_{ \epsilon \to 0^+} \big( g_{\epsilon}(y,y) - g_{\epsilon}(y,x) \big) = a(x-y). \label{E:10.20}
%\\
%&g_{\epsilon_N}(x,x) \sim g_{B_N}(x,x), \quad \text{ as $N \to \infty$}.  \label{E:10.21}
\end{align}
\end{lem}
\begin{proof}
Since $ g_{\epsilon}(\cdot,\cdot)$ is translation invariant, i.e., $g_{\epsilon}(x,y)= g_{\epsilon}(x+z,y+z)$, for all $x,y,z \in \mathbb{Z}^2$, it is enough to show \eqref{E:10.20} for $y=0$, $x \neq 0$. Recalling \eqref{E:2.3.5}, \eqref{E:2.3.7}, the potential kernel $a(\cdot)$ can be expressed in terms of the continuous-time heat kernel as
\begin{equation}
 \label{E:10.21}
a(x)=  \int_0^\infty \text d t \, \big( - \nabla_x q_t (0)\big), \quad \text{ for } x \in \mathbb{Z}^2.
\end{equation}
Indeed, \eqref{E:10.21} can be obtained as follows.  Writing $X_t=Z_{N_t}$, one notes that $q_t(0) - q_t(x) = E_0[p_{N_t}(0)-p_{N_t}(x)]$, with $p_n(\cdot)$ denoting the discrete-time kernel, cf. \eqref{E:2.1.p}, and therefore, applying Fubini's theorem and monotone convergence,
\begin{equation}
\label{E:10.22}
 \int_0^\infty \text d t \, |q_t(0) - q_t(x)|  \leq E_0\Big[ \sum_{k \geq 0} \tau_k |p_{k}(0)-p_{k}(x)|\Big]= \sum_{k \geq 0} |p_{k}(0)-p_{k}(x)| \stackrel{\eqref{E:2.1.2}}{<} \infty
\end{equation}
where $\tau_k$, $k \geq 0$ are i.i.d exponential random variables with parameter $1$ under $P_x$ (constituting $N_{\cdot}$, i.e. such that $N_t = \sup \{ \ell \geq 0 : \, \sum_{k \leq \ell} \tau_k \leq t \}$, for $t \geq 0$). By \eqref{E:10.22}, it follows that the integral on the right-hand side of \eqref{E:10.21} is well-defined and finite, and from \eqref{E:2.1} and dominated convergence, along with \eqref{E:10.22}, that it equals $a(x)$.

Returning to the Green function $g_{\epsilon}$ given by \eqref{eq:GF0}, performing the integral over $\xi(\epsilon)$ and applying Fubini (the relevant quantities are non-negative) yields
\begin{equation}
\label{E:10.23}
g_{\epsilon}(x,y)= E_x\Big[ \int_0^\infty \text d t \, e^{-\epsilon t} 1\{ X_t =y \} \Big] = \int_0^\infty \text d t \, e^{-\epsilon t} q_t(y-x).
\end{equation}
By \eqref{E:10.21} and \eqref{E:10.23}, we see that \eqref{E:10.20} follows at once if we show that
\begin{equation}
\label{E:10.24}
\int_1^\infty \text d t \, (1-e^{-\epsilon t}) \big(- \nabla_x q_t(0)\big) \rightarrow 0, \quad \text{ as } \epsilon \to 0^+
\end{equation}
(if one replaces the lower integration bound in \eqref{E:10.24} by $0$, the resulting expression is precisely $a(x)-(g_{\epsilon}(0,0) - g_{\epsilon}(0,x))$, and the integral from $0$ to $1$  is readily seen to vanish as $\epsilon \to 0^+$ by dominated convergence, since $| \nabla_x q_t| \leq 2$). Applying the gradient estimate \eqref{E:2.3.6}, and noting that the resulting error term, which decays as $t^{-2}$, is in $L^1([1,\infty))$, it suffices to show \eqref{E:10.24} with $\bar q_t$ in place of $q_t$. Finally, writing $c_x = c' |x|^2/2$, with $c'$ as appearing in \eqref{E:2.3.8}, and $F(\lambda)= e^{-c_x \lambda}$, we see that, for all $t \geq 1$,
$$
|\nabla_x \bar q_t(0)| = \frac{c}{t} \Big| F(0)-F\Big(\frac 1t\Big) \Big| \leq  \frac{c}{t^2} \sup_{s\leq \frac{1}{t}} |F'(s)| \leq \frac{\tilde c(x)}{t^2},
$$
and inserting this into the corresponding integral, \eqref{E:10.24} follows by dominated convergence.
\end{proof}

As consequence of Lemma \ref{L:10.20}, the pinned field $\varphi^{\,p}_{\cdot}$, cf. \eqref{E:2.16}, is also obtained by considering suitable increments of the massive field $\mathbf{P}_{\epsilon}^G$ defined in \eqref{E:massGFF} and removing the mass (cf. \eqref{E:2.17}):
\begin{equation}
\label{E:12.1}
\mathbf{E}^G[\varphi^{\,p}_x \varphi^{\,p}_y] = \lim_{\epsilon \to 0^+} \mathbf{E}^G_{\epsilon}[(\varphi_x^{\epsilon}-\varphi^{\epsilon}_0)(\varphi_y^{\epsilon} - \varphi^{\epsilon}_0)].
\end{equation}
To see this, one simply writes
\begin{equation}
\label{E:12.2}
\begin{split}
&\mathbf{E}^G_{\epsilon}[(\varphi^{\epsilon}_x-\varphi^{\epsilon}_0)(\varphi^{\epsilon}_y - \varphi^{\epsilon}_0)] \stackrel{\eqref{E:massGFF}}{=} g_{\epsilon}(x,y)- g_{\epsilon}(0,x) - g_{\epsilon}(0,y) + g_{\epsilon}(0,0)\\[0.2em]
& = \big(g_{\epsilon}(x,y) - g_{\epsilon}(x,x)\big) + \big(g_{\epsilon}(x,x) - g_{\epsilon}(0,x)\big) + \big(g_{\epsilon}(0,0) - g_{\epsilon}(0,y)\big).
\end{split}
\end{equation}
Letting $\epsilon \to 0$ and applying \eqref{E:10.20}, the first term in the second line of \eqref{E:12.2} converges to $-a(x-y) =- a(y-x)$, the second one to $a(x)$ and the third one to $a(y)$. In view of \eqref{E:2.16}, \eqref{E:12.1} follows.

%%%%%%%%%%%%%%%%
\section{Capacity, avoiding sets and Dirichlet b.c.}\label{S:2}

The following representation of the two-dimensional capacity (recall its definition in \eqref{E:2.12}) will prove useful below. In particular, it will naturally lead us to consider point processes avoiding a given set $K$ and killed when exiting a larger set $K'$. These can be obtained by suitable thinning operations, see Theorem \ref{T:2} below. For the time being, we focus on Dirichlet boundary condition - but see Remark \ref{R:cap},2) below. We write $B_N$, $N \geq 1$, for the Euclidean ball of radius $N$ around $0$ in $\mathbb{Z}^2$.
\begin{lem}
\label{L:4.cap} $(\emptyset \neq A \subset \subset \mathbb{Z}^2)$

\medskip
\noindent For any point $y \in A$,
\begin{equation}
\label{E:4.1}
\textnormal{cap}(A)= \frac{4}{\pi^2} \lim_{N \to \infty} ( \log N )^2 \sum_{x\in A} e_{A,B_N}(x) \,  P_x[H_y > T_{B_N}].
\end{equation}
\end{lem}
\begin{proof}
Fix a point $y \in A$.
Using \eqref{E:4.0}, \eqref{E:2.2} and \eqref{E:2.3}, one immediately sees that for all $x \in A$,
\begin{equation}
\label{E:4.2}
\lim_{N \to \infty} \big( g_{B_N}(y,y) - g_{B_N}(y,x) \big) = a(x-y).
\end{equation}
In view of the definition \eqref{E:2.12} and by virtue of \eqref{E:2.11}, one obtains, using \eqref{E:4.2},
\begin{equation}
\label{E:4.3}
\textnormal{cap}(A)=  \frac2\pi \lim_{N \to \infty} \log N \sum_{x\in A} \big(g_{B_N}(y,y) - g_{B_N}(y,x) \big)\,  P_x[\widetilde{H}_A > T_{B_N}].
\end{equation}
We will now show that the right-hand sides of \eqref{E:4.1} and \eqref{E:4.3} are equal. We tacitly assume henceforth that $N$ is large enough so that $\overline{A} \subset  B_N$, where $\overline{A}= A \cup \partial A$, $\partial A =\{ x\in \mathbb{Z}^2\setminus A: \, |x-y|=1 \text{ for some } y \in A  \}$.
By last exit decomposition for the walk killed outside of $T_{B_N}$, we know that for all $K \subset B_N$, $z \in \mathbb{Z}^2$, 
\begin{equation}
\label{E:4.4}
P_{z} [H_K < T_{B_N}] = \sum _{x\in K} g_{B_N} (z,x) P_x[\widetilde{H}_K \geq T_{B_N}] =   \sum _{x\in K} g_{B_N} (z,x) e_{K, B_N}(x),
\end{equation}
hence, we can rewrite
\begin{equation}
\label{E:4.5}
\begin{split}
&\Big( \sum_{x \in A} e_{A,B_N}(x) \Big)- e_{\{y \},B_N}(y) \\
&\quad \quad \quad = \frac{1}{g_{B_N}(y,y)} \sum_{x\in A}\big( g_{B_N}(y,y) - g_{B_N}(z,x)\big) e_{A,B_N}(x)\\
&\qquad \qquad \qquad+ \Big( \frac{1}{g_{B_N}(y,y)} \sum_{x\in A} g_{B_N}(z,x) e_{A,B_N}(x)\Big) - e_{\{y\},B_N}(y)\\
& \quad \quad \quad= \frac{1}{g_{B_N}(y,y)} \sum_{x\in A}\big( g_{B_N}(y,y) - g_{B_N}(z,x)\big) e_{A,B_N}(x) \\
&\qquad \qquad \qquad+   \frac{1}{g_{B_N}(y,y)} \Big( P_z[H_A < \infty] - \frac{g_{B_N}(y,y)}{g_{B_N}(z,y)}P_z[H_y < \infty] \Big),
\end{split}
\end{equation}
using \eqref{E:4.4} twice to obtain the last equality. Choosing $z=y$ and recalling that $A$ contains $y$, the last term in \eqref{E:4.5} vanishes and we deduce, applying the sweeping identity \eqref{E:2.15}, which implies that $e_{\{ y \},B_N}(y)= P_{e_{A,B_N}}[H_y < T_{B_N}]$, %for all $N \geq c$,
\begin{equation}
\label{E:4.6}
\begin{split}
&\sum_{x\in A}\big( g_{B_N}(y,y) - g_{B_N}(y,x)\big) e_{A,B_N}(x) \\
&\quad =  g_{B_N}(y,y) \bigg[ \Big( \sum_{x \in A} e_{A,B_N}(x) \Big)- e_{\{y\},B_N}(y) \bigg]
= g_{B_N}(y,y)  \sum_{x \in A}  e_{A,B_N}(x) P_x[H_y > T_{B_N}]. 
\end{split}
\end{equation}
Substituting \eqref{E:4.6} into \eqref{E:4.3}, and noting that $g_{B_N}(y,y) \sim g_{B_N}(0,0)$ and $g_{B_N}(0,0) = E_0[a(X_{T_{B_N}})] \sim \frac2\pi \log N$, on account of \eqref{E:2.3.0}, \eqref{E:2.3} and \eqref{E:4.0}, completes the proof.
\end{proof}

\begin{rmk}\label{R:cap} 1) Instead of working with the Green function, see \eqref{E:4.2}, one can also derive \eqref{E:4.1} with a martingale argument. Namely, assuming $x \neq y$ and $N$ is large enough such that $x,y \in B_N$, and introducing the stopping time $S= H_y \wedge T_{B_N}$, it follows from \eqref{E:2.1.1} that $(a(X_{s\wedge S} -y))_{s\geq 0}$ is a bounded martingale under $P_x$, hence by the stopping theorem,
\begin{equation*}
P_x[H_y> T_{B_N}] = \frac{a(x-y)}{E_x[a(X_S) \, | \, H_y> T_{B_N}]},
\end{equation*}
and together with \eqref{E:2.3.0}, this yields that
\begin{equation}
\label{E:4.7}
a(x-y)=   \lim_{N \to \infty} \frac2\pi (\log N) \, P_x[H_y> T_{B_N}], \quad \text{ for all $x,y \in B_N$}
\end{equation} 
(if $x = y$, equality also holds since both sides vanish). In particular, as follows from \eqref{E:4.7}, the presence of the square in \eqref{E:4.1} is due to the fact that each of the two probabilities appearing in the sum ``requires'' a normalizing factor of order $\log N$. \\ %This is best illustrated by taking for instance $A=\{ x,y \}$, $x \neq y$, a two-point set. 
\noindent 2) The representation \eqref{E:4.1} has an analogue in terms of ``massive'' quantities, see Proposition \ref{P:10.4} and Remark \ref{R:mass}, 1) below. \\
\noindent 3) It is instructive to compare \eqref{E:4.1} to the higher-dimensional setting, for which one defines  (see for instance \cite{La91}, Ch.2.2)
$$\text{cap}(A)= \lim_{N \to \infty}\sum_{x \in A} e_{A,B_N}(x)$$ (which of course vanishes in dimension $2$ by recurrence). As will soon become clear, foregoing the prefactor $(\frac2\pi \log N)^2$, the summation in \eqref{E:4.1} hints at a loop soup of bidirectional trajectories killed outside of $B_N$ whose backward part is conditioned to stay away from $A$ (until exiting $B_N$) and whose forward part is conditioned to avoid~$0$.~\hfill~$\square$
\end{rmk}
\medskip
We now define a suitable family of finite-volume interlacement point processes, which avoid a given set $K$ and are killed when exiting a larger set $K'$. To this effect, we introduce a point
\begin{equation}
\label{E:5.0}
x_{*} \notin \mathbb{Z}^2,
\end{equation}
and think of a killed random walk trajectory  as one reaching the (cemetery) point $x_*$ upon being killed and remaining in $x_*$ from then on forever. The spaces on which the intensity measures of the relevant processes are defined require a small amount of notation. We thus introduce the spaces $W$, resp. $W^+$ of bi-infinite,  resp. infinite $\mathbb{Z}^2 \cup\{ x_*\}$-valued right-continuous trajectories (and left-continuous at negative times in the bi-infinite case), whose forward and backward part make finitely many jumps before reaching $x_*$ and from then on remain equal to $x_*$. With regards to the next section, we also include in $W$, $W^+$ the set of $\mathbb{Z}^2$-valued trajectories (with no killing) escaping all finite sets in finite time (in the case of $W$, we require this to hold separately for the forward and backward part of any trajectory). We write $\mathcal{W}$, $\mathcal{W}^+$, $\theta_t$, $t \in \mathbb{R}$ (resp. $t\geq 0$) for the corresponding canonical $\sigma$-algebras and canonical shifts and $W^*$ for the quotient of $W$ modulo time-shift, with induced $\sigma$-algebra $\mathcal{W}^*$.  We denote by $\pi^*:W\to W^* $ the corresponding canonical projection. We write $W_A \subset W$ (resp. $W_A^+ \subset W^+$) for the (forward) trajectories entering $A$, 
for $A \subset \mathbb{Z}^2$, and $W_A^* = \pi^* (W_A)$.

We now define a measure $Q_A^{K,K'}$, for any pair $K,K'$ with $K \subset K' \subset \subset \mathbb{Z}^2$, and all $A$ satisfying $K \subset A \subset K' $, on $(W, \mathcal{W})$ supported on $W_A$ (in fact, supported on those trajectories in $W_A$ entering $A$ at time $0$) by
\begin{equation}
\label{E:5.1}
\begin{split}
&Q_A^{K,K'} [(X_{-t})_{t\geq 0} \in E^-, \, X_0 = x, \,(X_{t})_{t\geq 0} \in E^+ ]\\
&\qquad  =P_x[(X_{-t})_{0\leq t\leq T_{K'}} \in E^- \,| \, \widetilde{H}_A > T_{K'}] \, \rho_{A}^{K,K'}(x) \, P_x[(X_{t})_{0\leq t\leq T_{K'}} \in E^+ \, | \,H_K > T_{K'}]
\end{split}
\end{equation}
for all $ E^-,  E^+ \in \mathcal{W}^+$, $x \in A$, and with (cf. \eqref{E:2.14} for notation)
\begin{equation}
\label{E:5.2}
\rho_{A}^{K,K'}(x) \stackrel{\text{def.}}{=} e_{A,K'}(x) \,  P_x[H_K > T_{K'}] 1_A(x), \text{ for } x \in \mathbb{Z}^2.
\end{equation}
For $K =\emptyset$, we set $H_{\emptyset}= \infty$, so that $\{ H_K > T_{K'} \}= W$ and $Q_A^{\emptyset,K'}$, $\rho_{A}^{\emptyset,K'}$ are well-defined.
The following theorem asserts that the measures $Q_A^{K,K'}$ can be patched together in a consistent fashion. With hopefully obvious notation, for a measure $\mu$ on a space $(S,\mathcal{S})$ and $A \in \mathcal{S}$, we write $(1_A \mu ) (\cdot)\equiv \mu (\cdot \cap A)$ for the restriction of the $\mu$ to the set $A$ (its density with respect to $\mu$ is $1_A$).

\begin{thm} $(K \subset K' \subset \subset \mathbb{Z}^2)$ \label{T:2}

\medskip
\noindent There exists a unique finite measure $\nu^{K,K'}$ on $(W^*, \mathcal{W}^*)$ such that
\begin{equation}
\label{E:5.3}
1_{W_A^*}  \ \nu^{K,K'} = \pi^* \circ Q_A^{K,K'}, \text{ for all }K \subset A \subset K'. 
\end{equation}
\end{thm}
\begin{proof}
%One may prove \eqref{E:5.3} by following the steps of Theorem 1.1 in \cite{Sz10}. 
Let $K, K'$ and $A$ be as above, and define $W_{K,K'}\subset W$ as
\begin{equation}
\label{E:5.4.0}
W_{K,K'} = \{ w \in W ; \, H_{K}(w) > T_{K'}(w)\}
\end{equation}
(for bidirectional $w$, $H_{K}(w) = \inf\{t \in \mathbb{R}: X_t(w) \in K \}$ and $T_{K}=H_{\mathbb{Z}^2\setminus K}$, with the convention $\inf \emptyset = -\infty$)
which has full measure under $Q_A^{K,K'}$ (note in particular that $W_{K,K'} \subset W_{K'}$), and similarly let $W_{K,K'}^+\subset W^+$ by replacing $W$ with $W^+$ in \eqref{E:5.4.0}.
%Let $\bar{W}_{K}= W\setminus W_K$ denote the set of trajectories not entering $A$, and similarly $\bar{W}_{K}^+= W^+ \setminus W_K^+$.  
Consider the measure $1_{W_{K,K'}} Q_A^{\emptyset,K'}$. Since $K \subset A$, any (bidirectional) trajectory $w \in \text{supp}(Q_A^{\emptyset,K'})$ first enters in $K$ at some non-negative time, and possibly never, i.e. $H_K(w) \in[ 0,\infty]$. Hence, in view of \eqref{E:5.1}, the restriction of $Q_A^{\emptyset,K'}$ to $W_{K,K'}$ only affects the forward part of any trajectory in the support, yielding, for all $x \in K'$, $ E^-,  E^+ \in \mathcal{W}^+$,
 \begin{equation}
\label{E:5.4}
\begin{split}
&(1_{W_{K,K'}} Q_A^{\emptyset,K'})[(X_{-t})_{t\geq 0} \in E^-, \, X_0 = x, \,(X_{t})_{t\geq 0} \in E^+]\\
&\quad  \stackrel{\eqref{E:5.1}}{=} P_x[(X_{-t})_{0\leq t\leq T_{K'}} \in E^- \,| \, \widetilde{H}_A > T_{K'}] \, e_{A,K'}(x) \, \big(1_{W_{K,K'}^+}P_x\big)[(X_{t})_{0\leq t\leq T_{K'}} \in E^+ ]\\
&\quad \stackrel{\eqref{E:5.4.0}}{=} P_x[(X_{-t})_{0\leq t\leq T_{K'}} \in E^- \,, \, \widetilde{H}_A > T_{K'}]
\, P_x[(X_{t})_{0\leq t\leq T_{K'}} \in E^+ , \, H_K > T_{K'}]\\
&\quad \stackrel{\eqref{E:5.1}}{=} Q_A^{K,K'}[(X_{-t})_{t\geq 0} \in E^-, \, X_0 = x, \,(X_{t})_{t\geq 0} \in E^+].
\end{split}
\end{equation}
Taking projections, it follows from \eqref{E:5.4} that
\begin{equation}
\label{E:5.5}
1_{{W}_{K,K'}^*} (\pi^* \circ Q_A^{\emptyset,K'})= \pi^* \circ Q_A^{K,K'}, \quad \text{for all $A \subset K'$}.
\end{equation}
By Theorem 2.1 in \cite{Tei09}, we know there exists a unique measure $\nu^{\emptyset,K'}$ on $W^*$ such that 
\begin{equation}
\label{E:5.6}
1_{W_A^*}\cdot \nu^{\emptyset,K'} =  \pi^* \circ Q_A^{\emptyset,K'}, \quad \text{for all $A\subset K'$}
\end{equation} 
(note that due to the restriction of $X_{\cdot}$ to times $t\leq T_{K'}$ in \eqref{E:5.1}, we are effectively in the setup of a \textit{transient} weighted graph, see \cite{Tei09}, and also \cite{Sz12b} p.8, Example 2, by thinking of any trajectory in the support of $Q_A^{\emptyset,K'}$ as being killed when exiting $K'$, i.e. entering the absorbing state $x_*$, see \eqref{E:5.0}). The claim \eqref{E:5.3} follows from
\eqref{E:5.5} and \eqref{E:5.6} upon letting $\nu^{K,K'}= 1_{{W}_{K,K'}^*}  \nu^{\emptyset,K'} $.
\end{proof}

\begin{rmk}
Instead of invoking Theorem 2.1 of \cite{Tei09} (see also \cite{Sz10}, Theorem 1.1) for the case $K=\emptyset$ (no conditioning) and avoid the set $K$ by suitably restricting the measure, one can also work immediately with general $K$ and devise a more direct argument that follows the lines of their proofs. We briefly sketch the argument, which is instructive. The claim \eqref{E:5.3} follows readily if the following compatibility condition is satisfied: for all $A,A'$ with $K\subset A\subset A' \subset K'$, and all measurable $E \subset W_A^0$, the set of trajectories entering $A$ at time $0$,
\begin{equation}
\label{E:5.7}
Q^{K,K'}_A[E]= Q^{K,K'}_{A'}[\{ w \in W_A \cap W_{A'}^0 : \, \theta_{H_A}(w) \in E \}].
\end{equation}
For $K = \emptyset$, see for instance the proof of Theorem 1.1 in \cite{Sz10}, one essentially shows \eqref{E:5.7} by considering a trajectory $w\in W_A^0 \cap E$ and time-reversing its backward part starting at $A$ until its last visit to $A' \setminus A$. One then simply observes that this time-reversal can still be carried out without obstruction under the constraint that $K$ be avoided, since this only affects the \textit{forward} part of the trajectory, as seen from $A$. \hfill $\square$
\end{rmk}
One can naturally associate to the measure $ \nu^{K,K'}$ in \eqref{E:5.3} a family of Poisson point process indexed by an intensity parameter $u \geq 0$ as follows. Let
\begin{equation}
\label{E:5.10}
\begin{split}
\Omega= \Big\{ 
&\omega = \sum_{i\in I} \delta_{(w_i^*, u_i)}, \, w_i^* \in W^* \text{ for all } i \in I \text{ and $I \subseteq \mathbb{N}$, $u_i \geq 0$,} \\
&\text{ and } \omega(W_A^* \times [0,u])< \infty \text{ for all } A \subset \subset \mathbb{Z}^2 \text{ and } u\geq0 \Big\}
\end{split}
\end{equation}
be the space of locally finite point measures on $W^*\times [0,\infty)$, endowed with the $\sigma$-algebra~$\mathcal{F}$ generated by the evaluation maps $\omega \mapsto \omega(D^*)= \sum_{i \in I} 1\{ w_i^* \in D^* \}$ if $\omega = \sum_{i\in I} \delta_{w_i^*}$, for all $D^* \in \mathcal{W}^* \times \mathcal{B}$, where  $\mathcal{B}$ stands for the Borel $\sigma$-algebra on $[0,\infty)$. Since the infinite measure  $\nu^{K,K'}(\text dw^*)\text d u$ on $W^* \times [0,\infty)$ assigns finite mass to the sets $W^* \times [0,u]$, for any $u \geq 0$, we then classically introduce (see for instance \cite{Re87}) the probability measure $\mathbb{P}^{K,K'}$ on $(\Omega,\mathcal{F})$ 
 such that
 \begin{equation}
 \label{E:5.11}
 \begin{split}
&\text{under $\mathbb{P}^{K,K'}$, the law of $\omega$ is that of a Poisson point }\\
&\text{measure on $W^* \times [0,\infty)$ with intensity $\nu^{K,K'}(\text dw^*)\text d u$}.
\end{split}
\end{equation}
The process induces a field of occupation times on $\mathbb{Z}^2$, coming from collecting the cumulated time the various trajectories in the cloud spend at a particular vertex $x$: one defines $(L_{x,u})_{x\in \mathbb{Z}^2}$, for all $u > 0$, as
\begin{equation}
\label{E:5.12}
L_{x,u}(\omega)=\sum_{i \in I}\int_{-\infty}^\infty \text d t \, 1\{ w_i(t) = x, \, u_i \leq u \}, \text{ for } \omega =  \sum_{i\in I} \delta_{(w_i^*, u_i)},
\end{equation}
where $w_i \in W$ is any representant in the equivalence class of $w_i^*$, i.e. any trajectory satisfying $\pi^*(w_i) = w_i^*$. Note that $L_{x,u} = 0$ for all $x \notin K' \setminus K$, cf.  \eqref{E:5.1}. The corresponding interlacement set at level $u \geq 0$ (killed outside $K'$ and avoiding $K$) is then defined as
\begin{equation}
\label{E:5.13}
\mathcal{I}^u =\{x \in \mathbb{Z}^2:\, L_{x,u} > 0  \}, 
\end{equation} 
which is increasing in $u$ (since $L_{x,u}$ is) and 
\begin{equation}
\label{E:5.14}
\mathcal{V}^u = \mathbb{Z}^2 \setminus \mathcal{I}^u,  \text{ which satisfies } K\cup (K')^c \subset \mathcal{V}^u,
\end{equation}
is the corresponding vacant set. The law of $(1\{x \in \mathcal{I}^u\})_{x\in \mathbb{Z}^2}$ on $( \widetilde{\Omega} = \{0,1 \}^{\mathbb{Z}^2}, \widetilde{\mathcal{F}})$ (with canonical coordinates $\widetilde{Y}_{\cdot}$) under $\mathbb{P}^{K,K'}$ is denoted by $Q^{K,K'}_u$. It is characterized by the following property.
\begin{proposition}
 For all $A$ satisfying $K \subset A \subset K'$, one has
\begin{equation}
\label{E:5.15}
Q^{K,K'}_u [\widetilde{Y}_x= 0, \, x \in A] = \mathbb{P}^{K,K'}[\mathcal{I}^u \cap A = \emptyset]= \exp\Big\{-u \sum_{x \in A} \rho_{A}^{K,K'}(x)\Big\}, 
\end{equation}
with $ \rho_{A}^{K,K'}$ given by \eqref{E:5.2}.
\end{proposition}
\begin{proof} On account of \eqref{E:5.13} and \eqref{E:5.11}, one has, minding that $\omega(W_A^* \times [0,u])$ is the number of trajectories with label at most $u$ visiting the set $A$,
\begin{equation*}
\mathbb{P}^{K,K'}[\mathcal{I}^u \cap A = \emptyset] = \mathbb{P}^{K,K'}[\omega(W_A^* \times [0,u]) = 0] = e^{-u \nu^{K,K'}(W_A^*)} =  e^{-u Q_A^{K,K'}(W_A)},
\end{equation*}
using \eqref{E:5.3} for the last equality. The claim now follows from \eqref{E:5.1}.
\end{proof}
For a fixed set $A$ satisfying $K \subset A \subset K'$, it will be convenient to also introduce the point measure $\mu_{A,u}$, which comprises the (forward part of the) trajectories with label at most $u$ that visit the set $A$, from the time they first enter in $A$, i.e.
\begin{equation}
\label{E:5.30}
\mu_{A}^u (\omega) = \sum_{ i } \delta_{w_i^+} 1\{ u_i \leq u, \, H_A(w_i^*) < \infty\}, \text{ if } \omega=\sum_i \delta_{(w_i^*, u_i)},
\end{equation}
where $w_i^+ \in W^+$ is obtained from $w_i^* \in W_A^*$ by considering the unique element $w_i \in W$ such that $\pi^* (w_i) = w_i^*$ and $H_A(w_i)=0$, and restricting $w_i$ to its forward part. In view of \eqref{E:5.11}, and using \eqref{E:5.3}, \eqref{E:5.1} it readily follows that 
\begin{equation}
\label{E:5.31}
\begin{split}
&\text{$\mu_{A}^u$ is a Poisson random measure under $\mathbb{P}^{K,K'}$ }\\
&\text{with intensity $u P_{\rho_{A}^{K,K'}}[(X_{t})_{0\leq t\leq T_{K'}} \in \cdot \, | \,H_K > T_{K'}]$}.
\end{split}
\end{equation}

\section{Infinite-volume limit}\label{S:4}
We now describe the local limit $Q^{\alpha}$ of the trace of the random walk on the torus, see \eqref{E:1.5} and \eqref{E:1.3}, using a growing family of interlacements as defined in the previous section. We henceforth set
\begin{equation}
\label{E:6.1}
K=\{ 0 \}, \qquad K' = B_N, \text{ for some $N \geq 1$}
\end{equation}
and write  $\mathbb{P}^{0,N}$ (and similarly $\rho^{0,N}$, $Q^{0,N}$) for the measure $\mathbb{P}^{K,K'}$ defined by \eqref{E:5.11} with the choices \eqref{E:6.1}. We first investigate in Proposition  \ref{T:6.1} the behavior of the random set $\mathcal{V}^{u}$, see \eqref{E:5.14}, under $\mathbb{P}^{0,N}$ and with suitably tuned $u=u_N$ in the limit as $N \to \infty$. We then prove a (much) stronger convergence result at the level of the process $\omega$ itself, see Theorem \ref{T:LIMIT} and Corollary \ref{C:LIMIT} below. This will bring into play the \textit{tilted} interlacements of \cite{CPV15}.

\begin{proposition} $(\alpha > 0)$
\label{T:6.1}

\medskip
\noindent Let $u_N(\alpha)$, ${N \geq 1}$, be any squence of positive numbers satisfying
\begin{equation}
\label{E:6.2}
 u_N(\alpha) \sim  \frac{2}{\pi} \alpha \log^2N , \text{ as } N \to \infty.
\end{equation}
Then, the law of $(1\{ x\in \mathcal{V}^{u_N(\alpha)}\})_{x\in \mathbb{Z}^2}$ $($see \eqref{E:5.14}$)$ under $\mathbb{P}^{0,N}$ converges weakly as $N \to \infty $ towards $Q^{\alpha}$. %Moreover,
%\begin{equation}
%\label{E:6.3}
%\lim_{N\to \infty}  \mathbb{P}^{0,N}[ x\in \mathcal{V}^{u_N(\alpha)}] = \exp \{ - \frac \pi 2 \alpha a(x) \}, \quad  \text{for all $x \in \mathbb{Z}^2$}.
%\end{equation}
\end{proposition}

\begin{proof}
Let $A \subset \mathbb{Z}^2$ be a finite set containing the origin, and suppose henceforth that $N$ is large enough so that $A \subset B_N$. By \eqref{E:5.15}, we then know that 
\begin{equation}
\label{E:6.3}
\mathbb{P}^{0,N}[\mathcal{I}^{u_N(\alpha)}\cap A = \emptyset] = \exp \Big[ -u_N(\alpha) \sum_{x \in A} \rho_{A}^{0,N}(x) \Big].
\end{equation}
On the other hand, in view of \eqref{E:4.1}, we obtain that
\begin{equation*}
\begin{split}
&\lim_N \Big[ u_N(\alpha) \sum_{x \in A} \rho_{A}^{0,N}(x) \Big] \\
&\qquad \stackrel{\eqref{E:5.2} \eqref{E:6.2}}{=}  \frac{2}{\pi} \alpha  \lim_N  \Big[ \log^2 N \sum_{x \in A} e_{A,B_N}(x) \,  P_x[H_0 > T_{B_N}] \Big] = \frac{\pi}{2}\alpha \text{cap}(A). 
\end{split}
\end{equation*}
Substituting this into \eqref{E:6.3} completes the proof.
\end{proof}

%We now prove a (much) stronger convergence result at the level process $\omega$. 

Let
$ \nu^{0,\mathbb{Z}^2} $ be formally defined as in \eqref{E:5.3} with
\begin{equation}
\label{E:6.10}
Q_A^{0,\mathbb{Z}^2}[\, \cdot \, ] = \sum_{x\in A} a(x) \text{hm}_A(x) \widehat{P}_x [\, \cdot \, ],
\end{equation}
for $A \subset \subset \mathbb{Z}^d$ (one may in fact assume that $0 \in A$ since $Q_A^{0,\mathbb{Z}^2} = Q_{A\cup \{ 0\}}^{0,\mathbb{Z}^2}$), where $\widehat{P}_x$, $x \in \mathbb{Z}^2\setminus \{ 0\}$ is the law of the continuous time (with exponential holding times of parameter $1$) nearest-neighbor random walk on $ \mathbb{Z}^2$, endowed with edge weights 
\begin{equation}
\label{E:6.11}
\hat{c}_{x,y} = a(x)a(y)1\{ |x - y|=1\}, \quad x,y\in \mathbb{Z}^2
\end{equation} 
(note that $\hat{c}_{x,0} =0$). %We denote the corresponding canonical process by $\widehat{X}_t $, $t \geq 0$, whose
Its Markovian jump probabilities are given by 
\begin{equation}
\label{E:6.12}
\hat{p}_{x,y}= \frac{\hat{c}_{x,y}}{\hat{\lambda}_x}, \quad \text{with } \hat{\lambda}_x = \sum_{y: y\sim x} \hat{c}_{x,y} \stackrel{\eqref{E:2.1.1}}{=} 4a^2(x).
\end{equation}
The walk under $\widehat{P}_x$, $x \in \mathbb{Z}^2\setminus \{ 0\}$, is transient, hence the construction of the intensity measure $ \nu^{0,\mathbb{Z}^2} $ from \eqref{E:6.10} is within the realm of \cite{Tei09}, which treats random walks on general transient weighted graphs. 
We endow the space of Radon measures on $W^*$ with the topology of vague convergence. Thus, a sequence $\mu_N$ converges vaguely to $\mu$, denoted $\mu_N \stackrel{v}{\rightarrow} \mu$, as $N \to \infty$, if $\mu_N(f)\to \mu(f)$ for all continuous functions $f$ on $W^*$ with compact support.
\begin{thm} $(\alpha > 0)$
\label{T:LIMIT}

\medskip
\noindent For any sequence $u_N(\alpha)$, $N \geq 1$, satisfying \eqref{E:6.2},
\begin{equation}
\label{E:6.13}
\frac{2}{\pi} u_N(\alpha) \,  \nu^{0,N}   \stackrel{v}{\longrightarrow}  \alpha \, \nu^{0,\mathbb{Z}^2}, \quad \text{as } N \to \infty.
\end{equation}
\end{thm}
\begin{proof}
We start with a few reduction steps. By construction of $\nu^{0,N}$ and $\nu^{0,\mathbb{Z}^2}$, cf. \eqref{E:5.3}, and since $u_N(\alpha)$ satisfies \eqref{E:6.2}, it suffices to show that for all finite $A\subset \mathbb{Z}^2$ containing the origin, 
\begin{equation}
\label{E:6.13}
\Big(\frac{2}{\pi} \log N \Big)^2 Q_A^{0,N}(E) \stackrel{v}{\rightarrow}  Q_A^{0,\mathbb{Z}^2}(E), \quad \text{ as $N \to \infty$}, 
\end{equation}
where $E$ is an arbitrary event of the following form: writing, for $w \in W^+ $, $T_i = T_i(w)$, $i \geq 1$ for its successive jump times, and $N_t = N_t(w)$ for the number of jumps up to time $t$, one has, for some $n \geq 0$, $(x_0, x_1,\dots,x_n)$ a nearest neighbor path on $\mathbb{Z}^2 \setminus\{ 0\}$ with $x_0 \in A$, $t > 0$ and $0< t_1< \dots < t_n < t$,
$$
E=\{  N_t =n , \, X_0 = 0, \, X_{T_i}= x_i, \, T_i \in [t_i, t_i + \text{d} t_i), \, 1\leq i \leq n \},
$$
where the precise meaning of $T_i \in [t_i, t_i + \text{d} t_i)$ is to consider some measurable subset $A_i \subset (0,\infty)$ and integrate $t_i$ over this set. Recall that $Z_n$, $n \geq 0$, is the discrete skeleton of $X_{\cdot}$, defined by $X_t = Z_{N_t}$, for $t \geq 0$. With $N_t =\sup\{n\geq0:\, T_1+ \dots+T_n \leq t\}$, we thus have $Z_n = X_{T_n}$ for all $n \geq 0$. Fix an event $E$ as above. By definition of $Q_A^{0,N}$ in \eqref{E:5.1}, \eqref{E:5.2}, minding that $\frac2\pi \log N e_{A,B_N}(x_0) \to \text{hm}_A(x_0)$, see \eqref{E:2.11}, the claim \eqref{E:6.13} follows at once if one shows
\begin{equation}
\label{E:6.14}
\begin{split}
&\lim_{N \to \infty} \Big(\frac{2}{\pi} \log N  \Big)  P_{x_0}[ Z_{i\wedge T_{B_N}(Z)}= x_i,  1\leq i \leq n, \,H_0(Z) > T_{B_N}(Z)] \\
&\quad = a(x_0) \widehat{P}_{x_0} [ Z_i= x_i, \, 1\leq i \leq n], 
\end{split}
\end{equation}
where $T_{B_N}(Z)= \inf \{ n \geq 0: \, Z_n \notin B_N\}$ is the exit time for the discrete chain $Z$ and similarly for $H_0$. On account of \eqref{E:6.11}, \eqref{E:6.12}, the right-hand side of \eqref{E:6.14} equals
\begin{equation}
\label{E:6.15}
 a(x_0)\hat p_{x_0,x_1} \hat p_{x_1,x_2}\dots \hat p_{x_{n-1},x_n} = \Big(\frac{1}{4}\Big)^n a(x_n).
\end{equation}
On the other hand, applying the simple Markov property at time $n$, we can rewrite the probability on the left-hand side of \eqref{E:6.14} as
\begin{equation}
\label{E:6.16}
\begin{split}
 &P_{x_0}[ Z_{i}= x_i,  1\leq i \leq n, T_{B_N}(Z) > n, \,  \,H_0(Z) > T_{B_N}(Z)] \\[0.2em]
 &\qquad =   P_{x_0}[ Z_{i}= x_i,  1\leq i \leq n] P_{x_n}[H_0(Z) > T_{B_N}(Z)] = \Big(\frac{1}{4}\Big)^n P_{x_n}[H_0 > T_{B_N}].
 \end{split}
\end{equation}
The claim \eqref{E:6.14} then follows from \eqref{E:6.15} and \eqref{E:6.16}, together with \eqref{E:2.3} and the fact that
\begin{equation}
\label{E:6.16.0}
P_{x_n}[H_0 > T_{B_N}] = \frac{a(x_n)}{\frac2\pi \log N + O((\log N)^{-1})}, \quad \text{ as } N \to \infty,
\end{equation}
which follows readily from the optional sampling theorem applied to the martingale $a(X_{H_0 \wedge T_{B_N}})$, cf. \eqref{E:4.7}. 
\end{proof}

We now denote by $\mathbb{P}^{0,\mathbb{Z}^2}$ the law on $\Omega$, cf. \eqref{E:5.10}, of the Poisson random measure with intensity $\nu^{0,\mathbb{Z}^2}(\text dw^*)\text d u$.
These are the tilted interlacements of \cite{CPV15}. For an element $\omega = \sum_{i\in I} \delta_{(w_i^*, u_i)} \in \Omega$, we further write $\omega_u =  \sum_{i\in I: \, u_i \leq u} \delta_{w_i^*}$ for the point measure obtained by retaining the trajectories with label at most $u$ (and removing their label). Theorem \ref{T:LIMIT} then has the following consequence.

\begin{corollary} $(\alpha > 0, \, (u_N(\alpha))_N \text{ as in \eqref{E:6.2}})$
\label{C:LIMIT}
\begin{equation}
\label{C:6.17}
\text{The law of $\omega_{u_N(\alpha)}$ under $\mathbb{P}^{0,N}$ converges weakly towards the law of $\omega_{\frac\pi2 \alpha}$ under $\mathbb{P}^{0,\mathbb{Z}^2}$.}
\end{equation}
\end{corollary}

\begin{proof}
Let %$\mathbb{P}^{0,N}_\alpha = \mathbb{P}^{0,N} \circ \omega_{u_N(\alpha)}^{-1}$ and $\mathbb{P}_{\alpha}^{0,\mathbb{Z}^2}= \mathbb{P}^{0,\mathbb{Z}^2} \circ \omega_{\frac\pi2 \alpha}^{-1}$ and let 
$\Psi_N(f)= \mathbb{E}^{0,N} [e^{-\omega_{u_N(\alpha)}(f)}]$ and $\Psi(f)= \mathbb{E}^{0,\mathbb{Z}^2} [e^{-\omega_{\frac\pi2 \alpha}(f)}]$, for non-negative measurable $f:W^*\to \mathbb{R}$ denote the relevant Laplace functionals. By \cite{Re87}, Prop. 3.19, p.153, it suffices to show that $\Psi_N(f) \to \Psi(f)$ for any non-negative, continuous $f$ with compact support. But since $\omega_{u_N(\alpha)}$ is a Poisson random measure, one has (see \cite{Re87}, Prop. 3.6, p.130),
\begin{equation}
\label{E:6.18}
\Psi_N(f) = \exp \bigg\{ - u_N(\alpha) \int_{W^*} (1-e^{-f(w^*)}) \nu^{0,N}(\text dw^*)  \bigg\},
\end{equation}
and a similar formula with $\frac\pi2 \alpha$ in place of $u_N(\alpha)$ and $\nu^{0,\mathbb{Z}^2}$ in place of $ \nu^{0,N}$ holds for $\Psi(f)$. But since $1-e^{-f}$ is continuous with compact support whenever $f$ is, the desired convergence follows immediately from \eqref{E:6.13}.
\end{proof}

We conclude with a result concerning local times tailored to the purposes of the next section. One associates to $\mathbb{P}^{0,\mathbb{Z}^2}$ a field of local times $(L_{x,u})_{x\in \mathbb{Z}^2}$, defined as in \eqref{E:5.12}, for any $u>0$. The space $[0,\infty)^{\mathbb{Z}^2}$ is endowed with the product topology, and convergence in distribution in the following statement is meant in the sense of convergence in law of all finite-dimensional marginals.
\begin{lem} $(\alpha > 0, \, u_N(\alpha) \text{ as in \eqref{E:6.2}})$
\label{L:6.20}
\begin{equation} 
\begin{split}
\label{E:6.21}
&\text{$(L_{x,u_N(\alpha)})_{x\in \mathbb{Z}^2}$ under $\mathbb{P}^{0,N}$ converges in}\\
&\text{distribution towards $(L_{x,\frac\pi2 \alpha})_{x\in \mathbb{Z}^2}$ under $\mathbb{P}^{0,\mathbb{Z}^2}$.}
\end{split}
\end{equation}
\end{lem}

\begin{proof}
This is essentially a consequence of Theorem \ref{T:LIMIT}. Let $V: \mathbb{Z}^2\to \mathbb{R}_+$ be compactly supported. Defining $\xi : W^* \to [0,\infty)$ by
\begin{equation}
\label{E:6.22}
\xi(w^*) = \sum_x V(x) \int_{-\infty}^\infty 1\{ w(s)=x\} \text d s, \quad \text{ for any $w \in W$ s.t. $\pi^*(w)=w^*$},
\end{equation}
%and letting $\xi = \lim_{n\to \infty} \uparrow \xi_n$, 
we see from \eqref{E:5.12} that
\begin{equation}
\label{E:6.23}
\begin{split}
&\mathbb{E}^{0,N}\Big[\exp \Big\{ - \sum_x V(x)L_{x,u_N(\alpha)} \Big\} \Big] = \mathbb{E}^{0,N}\big[\exp \big\{ - \omega_{u_N(\alpha)}(\xi) \big\} \big] \\
&\quad= \psi_N (\xi) \stackrel{\eqref{E:6.18}}{=} \exp \big\{  - u_N(\alpha) \nu^{0,N} \big(1-e^{-\xi}\big)\big\}.
\end{split}
\end{equation}
Assuming that $A=\text{supp}(V)$, it follows from \eqref{E:6.22}, \eqref{E:5.31} and \eqref{E:5.2} that
\begin{equation}
\label{E:6.24}
u_N(\alpha) \nu^{0,N} \big(1-e^{-\xi}\big)= u_N(\alpha) E_{e_{A,B_N}} \Big[ \Big(1- e^{-\int_0^{T_{B_N}} V(X_t) \text d t} \Big) 1\{H_0 > T_{B_N}\}\Big].
\end{equation}
One then shows by a calculation similar to \eqref{E:6.14}, noting that $T_{B_N} \nearrow \infty$, that for all $x\in A$, 
$$
 (\log N) \cdot E_{x} \Big[ \Big(1- e^{-\int_0^{T_{B_N}} V(X_t) \text d t} \Big) 1\{H_0 > T_{B_N}\}\Big] \stackrel{N}{\longrightarrow} a(x) \cdot \widehat{E}_{x} \Big[ 1- e^{-\int_0^{\infty} V(X_t) \text d t} \Big],
$$
so that, in view of \eqref{E:6.23} and \eqref{E:6.24},  $\psi_N (\xi)$ converges as $N \to \infty$ towards
\begin{equation}
\label{E:6.25}
 \exp \Big\{  - \frac{\pi}{2}\alpha \sum_{x \in A} \text{hm}_A(x) a(x)  \widehat{E}_{x} \Big[ 1- e^{-\int_0^{\infty} V(X_t) \text d t} \Big]\Big\} = \mathbb{E}^{0,\mathbb{Z}^2}\Big[\exp \Big\{ - \sum_x V(x)L_{x,\frac\pi2 \alpha} \Big\} \Big],
\end{equation}
where the last step in \eqref{E:6.25} follows from \eqref{E:6.10} and a computation analogous to \eqref{E:6.23}. %and the claim follows by approximating $\xi$ in \eqref{E:6.22} monotonously (assuming $V \geq0$, and otherwise decomposing $V$ into positive and negative parts) by functions $\xi_n$ with compact support in $W^*$ (for instace by truncating both forward and backward part of a trajectory $w$ after $n$ steps), and using \eqref{E:6.13} and monotone convergence in \eqref{E:6.23} to deduce that  $\psi_N (\xi) \to \psi(\xi)$, which yields \eqref{E:6.21}.
\end{proof}

\section{Some links to the pinned field}\label{S:5}

As we now explain, one can naturally associate the Poisson point process $\mathbb{P}^{K,K'}$ in \eqref{E:5.11} to the excursions of a single Markov chain on a graph with vertex set $K\cup\{ x_*\}$, with $x_* \notin K' $ (cf. \eqref{E:5.0}). This is the content of Lemma \ref{L:7.1}, which will then be used for the choices \eqref{E:6.1}, along with the classical Ray-Knight theorem, to deduce isomorphisms relating the field of local times $(L_{x,u})_{x\in \mathbb{Z}^2}$, under $\mathbb{P}^{0,N}$ and $\mathbb{P}^{0,\mathbb{Z}^2}$ to certain pinned Gaussian fields, see Theorems \ref{T:7.12}, \ref{T:7.12}' and \ref{T:7.29} below.

Fix some set $K'$ with $\emptyset \neq K' \subset \subset \mathbb{Z}^2$. For the Markov chain we will consider, $x_*$ will play the role of a reference state outside $K'$ rather than a cemetery state (in particular, it is not absorbing). Define the conductances $c_{x,y} = c_{x,y}(K')  \geq 0$, $x,y \in K'\cup\{ x_*\}$ as follows:
\begin{equation}
\label{E:7.1}
\begin{split}
&c_{x,y}=1 \text{ if $x,y \in K'$, $|x-y|=1$ },\\
& c_{x,x_{*}} = c_{x_{*},x}=\sum_{y\in \mathbb{Z}^2 \setminus K': \, |x-y|=1 } 1, \text{ if } x\in \partial_{\text{int}} K' \ \big(  \stackrel{\text{def.}}{=} \partial (\mathbb{Z}^2 \setminus K')\big)\\
& c_{x,y}=0 \text{ otherwise.}
\end{split}
\end{equation}
In particular $c_{y,x}=c_{x,y}$ for all $x,y$. We also write
\begin{equation}
\label{E:7.1.1}
\lambda_x = \sum_y c_{x,y}, \quad \text{for all }x \in K'\cup\{ x_*\},
\end{equation}
and note that $\lambda_x = 4$ for all $x \in K'$. We consider the continuous time random walk on $K'\cup\{ x_*\}$ attached to these conductances with jump rates equal to $1$. Its canonical law started at $x$ is denoted by $\overline{P}^{K'}_x$, the corresponding expectation by $\overline{E}^{K'}_x$ and the canonical process by $\overline{X}_{t}^{K'}$, $t \geq 0$. For later reference, we observe that, for any point $x \in K'$,
\begin{equation}
\label{E:7.1.2}
\begin{split}
&\text{the law of $(\overline{X}_{t}^{K'})_{0\leq t < T_{K'}}$ under $\overline{P}^{K'}_x$ is the}\\
&\text{same as the law of $({X}_{t})_{0\leq t < T_{K'}}$ under ${P}_x$}.
\end{split}
\end{equation}
 The local time of the walk at $x$ is defined as
\begin{equation}
\label{E:7.2}
\bar{\ell}_t^{K',x} =\int_0^ t \, \text d s \,  1\{ \overline{X}_{s}^{K'}=x\}, \quad \text{ for } x\in K'\cup\{ x_*\}, \, t \geq 0,
\end{equation}
which is continuous and increasing to infinity as $t \to \infty$ since $\overline{X}_{\cdot}^{K'}$ is recurrent. In particular, under $\overline{P}^{K'}_{x_*}$, one has two almost surely infinite sequences $R_n$, $n \geq 1$ and $D_n$, $n \geq 1$ of successive departure times from $x_*$ and return times to $x_*$, such that $R_1=0 < D_1 < R_2< D_3< \dots$, and one can correspondingly partition the trajectory of $\overline{X}_{\cdot}^{K'}$ on $K'$ into excursions from $x_*$, given by $(\overline{X}_{D_n + t}^{K'})_{0\leq t \leq R_{n+1}-D_n}$, $n \geq 1$. By extending their value to be $x_*$ for all times $t \geq R_{n+1}-D_n$, these (random) excursions can naturally be seen to take values in the space $W^+$, cf. below \eqref{E:5.0} (recall $W^+$ contains all right-continuous, $\mathbb{Z}^2$-valued trajectories with finitely many jumps, which at a finite times reach $x_*$ and from that time on remain in $x_*$). For $s \geq 0$, we then define the random point measure on $(W^+,\mathcal{W}^+)$
collecting the excursions starting before time $s$ as
\begin{equation}
\label{E:7.3}
\bar{\mu}^s=\sum_{n \geq 1} 1\{D_n < s \} \delta_{(\overline{X}_{D_n + t}^{K'})_{0\leq t \leq R_{n+1}-D_n}},
\end{equation}
as well as 
\begin{equation}
\label{E:7.4}
\bar{\mu}_A^s = \Theta_{H_A} \circ (1\{ H_A<\infty\} \bar{\mu}^s), \quad \text{for $A \subset K'$}
\end{equation}
(also a random point measure on $(W^+,\mathcal{W}^+)$), where $\Theta_{H_A} \circ \nu = \sum_{i}\delta_{\theta_{H_A}(w_i)}$ for $\nu=\sum_{i}\delta_{w_i} $ with $w_i \in W_A^+$. In words, $\bar{\mu}_A^s$ selects the excursions in the  support of $\bar{\mu}^s$ that enter $A$ and only keeps track of their trajectories after they first enter $A$. Note that $\bar{\mu}^s= \bar{\mu}^s_{K'}$. Finally, recall the definition of $\mu_A^u$ from \eqref{E:5.30}, \eqref{E:5.31}. The following lemma relates these two quantities. Note that $K$ could very well be the empty set in what follows.

\begin{lem}$(u>0,  \,  K \subsetneq  A \subset K' \subset \subset \mathbb{Z}^2)$
\label{L:7.1}
\medskip
\noindent Let 
\begin{equation}
\label{E:7.5}
\sigma_u^{K'} = \inf \{ t \geq 0: \, \bar{\ell}_t^{{K'},x_*}  > \frac{\lambda_{x_*}}{4} u \}.
\end{equation}
%(the right-continuous inverse of $t\mapsto \bar{\ell}_t^{N,x_*}$). 
Then the law of $\bar{\mu}_A^{\sigma_u^{K'}}$ under $\overline{P}^{K'}_{x_*}[\, \cdot \, | \, \bar{\ell}_{\sigma_u^{K'}}^{{K'},x} =0, \, x \in K]$ equals that of $\mu_A^u$ under $\mathbb{P}^{K,K'}$. 
\end{lem}

\begin{rmk}
The event $\{ \bar{\ell}_{\sigma_u^{K'}}^{{K'},x} =0, \, x \in K \}= \{ \overline{X}_{t}^{K'} \notin K, \, 0\leq t \leq \sigma_u^{K'}\}$ has positive probability under $\overline{P}^{K'}_{x_*}$, for every  $u > 0$. In particular there is no issue in defining the conditional law above. \hfill $\square$
\end{rmk}
\begin{proof}
By \eqref{E:5.31} it is enough to check that $\bar{\mu}_A^{\sigma_u^{K'}}$ is a Poisson random measure under $\overline{P}^{K'}_{x_*}[\, \cdot \, | \,\bar{\ell}_{\sigma_u^{K'}}^{{K'},x} =0, \, x \in K]$ with intensity $u P_{\rho_{A}^{K,K'}}[(X_{t})_{0\leq t\leq T_{K'}} \in \cdot \, | \,H_K > T_{K'}]$. Let $\tau_i = D_{i}-R_i$, $i \geq 1$, be the time spent in $x_*$ before the $i$-th excursion. By construction, $\tau_i$, $i \geq 1$, are i.i.d exponential random variables with parameter $1$. Moreover, by \eqref{E:7.5}, the event $\{D_n < \sigma_u^{K'} \} $ is the same as $\{ \tau_1 +\dots+\tau_n < 4 u/ \lambda_{x_*}\}$. Hence, the number of excursions in the  support of the measure $\bar{\mu}^{\sigma_u^{K'}}$ in \eqref{E:7.3}, which equals $\sup_n \{D_n < \sigma_u^{K'} \}$ (with the convention $\sup \emptyset = 0$) is a Poisson variable with parameter $u\lambda_{x_*}/4$. Moreover, applying the strong Markov property at times $R_n$, we see that the excursions $(\overline{X}_{D_n + t}^{K'})_{0\leq t \leq R_{n+1}-D_n}$ are independent, and distributed as $\overline{P}_{\lambda^{K'}}[(\overline{X}_{ t}^{K'})_{0\leq t \leq T_{K'}} \in \cdot]$, where $\lambda^{K'}$ is the measure supported on $\partial_{\text{int}}K'$ defined as
\begin{equation}
\label{E:7.6}
\lambda^{K'}(x)=\frac{c_{x_*,x}}{\sum_{y \sim x_*}c_{x_*,y}}\stackrel{\eqref{E:7.1},\eqref{E:7.1.1}}{=} \frac1{\lambda_{x_*}}\sum_{y\in \mathbb{Z}^2 \setminus K': \, |x-y|_1=1 } 1 \stackrel{\eqref{E:2.14}}{=} \frac{4}{\lambda_{x_*}} e_{K', K'}(x).
\end{equation}
On account of \eqref{E:7.1.2}, we have thus obtained that $\bar{\mu}^{\sigma_u^{K'}}$ is a Poisson random measure on $(W^+, \mathcal{W}^+)$ with intensity
\begin{equation}
\label{E:7.7}
(u\lambda_{x_*}/4) \cdot \overline{P}_{\lambda^{K'}}[(\overline{X}_{ t}^{K'})_{0\leq t \leq T_{K'}} \in \cdot] \stackrel{\eqref{E:7.6}}{=} u {P}_{ e_{K', K'}}[({X}_{ t \wedge T_{K'}})_{t \geq 0} \in \cdot].
\end{equation}
From \eqref{E:7.7}, \eqref{E:7.4}, and Theorem 2.1 of \cite{Tei09} (we can view \eqref{E:7.7} in a transient setup with $x_*$ as absorbing state), we infer that $\bar{\mu}^{\sigma_u^{K'}}_A$ is a Poisson random measure with intensity $u {P}_{ e_{A, K'}}[({X}_{ t \wedge T_{K'}})_{t \geq 0} \in \cdot]$ under $\overline{P}^{K'}_{x_*}$. Finally, since $A$ contains $K$, we split the measure $\bar{\mu}^{\sigma_u^{K'}}_A$ as
\begin{equation}
\label{E:7.8}
\bar{\mu}^{\sigma_u^{K'}}_A = \gamma_{K} + \tilde{\gamma}
\end{equation}
where $\gamma_{K}$ is the point measure that collects the trajectories in the support of $\bar{\mu}^{\sigma_u^{K'}}_A $ which enter $K$ before exiting $K'$ and $\tilde{\gamma}$ contains those for which $H_K= \infty$. The random measures $\gamma_{K}$ and $\tilde{\gamma}$ are independent and $\{ \bar{\ell}_{\sigma_u^{K'}}^{{K'},x} =0, \, x \in K \} = \{ \gamma_{K}(W^+)=0\}$. Thus, the law of $\bar{\mu}_A^{\sigma_u^{K'}}$ under $\overline{P}^{K'}_{x_*}[\, \cdot \, | \, \bar{\ell}_{\sigma_u^{K'}}^{{K'},x} =0, \, x \in K]$ is the same as that of $\tilde{\gamma}$ under $\overline{P}^{K'}_{x_*}$, which is a Poisson random measure with intensity
\begin{equation*}
u {P}_{ e_{A, K'}}[({X}_{ t \wedge T_{K'}})_{t \geq 0} \in \cdot, \, H_0 = \infty] \stackrel{\eqref{E:5.31}}{=} u P_{\rho_{A}^{0,K'}}[(X_{t})_{0\leq t\leq T_{K'}} \in \cdot \, | \,H_0 > T_{K'}].
\end{equation*}
\end{proof}

We now specialize to the case $K,K'$ as in \eqref{E:6.1}, in which we are ultimately interested, and refer the reader to Remark \ref{R:7.29}, 3) below for a more general version of the following result. We routinely write $\overline{P}^{N}_x$, $\overline{X}^{N}_t$, etc. in what follows when referring to the Markov chain defined below \eqref{E:7.1.1} with $K'=B_N$. Recall the law $\mathbf{P}_N^G$ of the Gaussian free field with zero boundary condition outside $B_N$, cf. \eqref{E:7.9}, and define, for $h \in \mathbb{R}$,
\begin{equation}
\label{E:7.10}
\Phi_x(h) = \varphi_x + P_x[H_0<T_{B_N}](h-\varphi_0) = \Phi_x(0) + P_x[H_0<T_{B_N}]h , \quad x \in \mathbb{Z}^2, 
\end{equation}
and set
\begin{equation}
\label{E:7.10.0}
\widetilde{\varphi}_x = \Phi_x(0), \quad x \in \mathbb{Z}^2
\end{equation}
(so that $\widetilde{\varphi}_0 = 0$). Observing that $g_{B_N}(x,0)= P_x[H_0 < T_{B_N}]g_{B_N}(0,0)$, we deduce that $(\widetilde{\varphi}_x)_{x\in \mathbb{Z}^2}$ is orthogonal to (and hence independent of) $\varphi_0$ under $\mathbf{P}_N^G$, i.e. $\mathbf{E}_N^G[\widetilde{\varphi}_x \varphi_0]=0$ for all $x \in \mathbb{Z}^2$. Thus, the law of $(\Phi_x(h))_{x \in \mathbb{Z}^2}$, $h \in \mathbb{R}$, is a choice of regular conditional distribution for the field $(\varphi_x)_{x\in \mathbb{Z}^2}$ given its value $\varphi_0=h$. Moreover, using \eqref{E:2.13.0}, cf. also the calculation leading to Lemma 1.2 of \cite{RoS13}, we obtain that $\widetilde\varphi$ is a centered Gaussian field with covariance
 \begin{equation}
 \label{E:7.11}
\mathbf{E}^G_N[\widetilde{\varphi}_x \widetilde{\varphi}_y] = g_{B_{N}\setminus \{ 0\}}(x,y), \quad \text{ for all }x ,y \in \mathbb{Z}^2.
 \end{equation}

The above Gaussian field(s) can be linked to the field of local times $(L_{x,u})_{x\in \mathbb{Z}^2}$ associated to the interlacement $\mathbb{P}^{0,N}$, see \eqref{E:5.12}, as follows.
 
\begin{thm} (Pinned isomorphism theorem, finite volume) 
\label{T:7.12}

\medskip
\noindent
For all integers $N \geq 1$, $u\in (0, \infty)$, and with $\widetilde{\varphi}$ as defined in \eqref{E:7.10},
\begin{equation}
\label{E:7.12}
\begin{split}
&\text{the law of }\Big(L_{x,u} + \frac12 \widetilde\varphi_x^{\,2}\Big)_{x\in B_N} \text{ under } \mathbb{P}^{0,N} \otimes \mathbf{P}_N^G,\\
&\text{is the same as the law of } \Big( \frac12 ( \widetilde\varphi_x + h_x^N(u))^2 \Big)_{x\in B_N} \text{ under $\mathbf{P}_N^G$},
\end{split}
\end{equation}
where
\begin{equation}
\label{E:7.13}
h_x^N(u) = P_x[H_0 >T_{B_N}] (2u)^{1/2}.
\end{equation}
\end{thm}
\begin{proof}
Our starting point is the generalized second Ray-Knight theorem, see for instance Theorem 8.2.2 in \cite{MR06} or Theorem 2.17 in \cite{Sz12b}, applied to the (recurrent) Markov chain $\overline{X}_{\cdot}^N$, which yields that
\begin{equation}
\label{E:7.14}
\Big(\bar\ell_{\sigma_u^N}^{N,x} + \frac12 \varphi_x^{\,2}\Big)_{x\in B_N} \text{ under } \overline{P}_{x_*}^{N} \otimes \mathbf{P}_N^G, \text{ has the same law as }( \psi_x^u )_{x\in B_N} \text{ under $\mathbf{P}_N^G$},
\end{equation}
where $ \psi_x^u = \frac12 ( \varphi_x + \sqrt{2u})^2$, and $\varphi$ is the Gaussian field defined in \eqref{E:7.9}. We first consider the conditional law of $\psi_{\cdot}^u$ given $\psi_0^u$. From the discussion following \eqref{E:7.10}, we know that $\mathbf{P}_N^G[\psi_{\cdot}^u \in \cdot | \varphi_0] = \mathbf{P}_N^G[ \frac12 ( \Phi_{\cdot}(\varphi_0) + \sqrt{2u})^2 \in \cdot]$ $\mathbf{P}_N^G$-a.s., with $\Phi_{\cdot}$ as defined in \eqref{E:7.10}.
Slight care is needed when conditioning on $\psi_0^u$ instead, by which one loses the information on $\text{sign}(\varphi_0 + \sqrt{2u})$ due to the square. By first conditioning on $\varphi_0$, one obtains, $\mathbf{P}_N^G$-a.s.,
\begin{equation}
\label{E:7.15}
\begin{split}
&\mathbf{P}_N^G[\psi_{\cdot}^u \in \cdot | \psi_0^u]= \mathbf{P}_N^G\big[ \mathbf{P}_N^G[\psi_{\cdot}^u \in \cdot  | \varphi_0 ] \, \big| \psi_0^u\big]\\
&\quad= F^{+}(\psi_0^u) \cdot \mathbf{E}_N^G[1\{\varphi_0 \geq -\sqrt{2u} \}\, | \, \psi_0^u] + F^{-}(\psi_0^u)  \cdot \mathbf{E}_N^G[1\{\varphi_0 < -\sqrt{2u} \}\, | \, \psi_0^u] 
\end{split}
\end{equation} 
(the dot stands for any fixed measurable subset of $\overline{\Omega}$), where
\begin{equation}
\label{E:7.16}
F^{\pm} (t) = \mathbf{P}_N^G\bigg[ \frac12 \Big( \Phi_{\cdot}\big( \pm\sqrt{2t} - \sqrt{2u}\, \big) + \sqrt{2u}\Big)^2 \in \cdot\bigg].
\end{equation}
Let $\varepsilon > 0$ and consider the conditional probability $\mathbf{P}_N^G[\psi_{\cdot}^u \in \cdot \, | \, \psi_0^u < \varepsilon]$. Substituting \eqref{E:7.15}, bounding the (non-negative) continuous functions $F^{\pm}(\cdot)$ from above by their supremum over $[0,\varepsilon]$ (and similarly from below), and noting, with $\phi(\lambda)= \mathbf{P}_N^G[\varphi_0 \leq \lambda]$, abbreviating $\lambda= -\sqrt{2u}$, $\tilde\varepsilon = \sqrt{2\varepsilon}$, that 
$$
 \mathbf{P}_N^G[\varphi_0 \geq -\sqrt{2u} \, | \, \psi_0^u  < \varepsilon] = \frac{\phi(\lambda +\tilde\varepsilon)-\phi(\lambda)}{\phi(\lambda+\tilde\varepsilon)- \phi(\lambda-\tilde\varepsilon)} \stackrel{\varepsilon \to 0^+}{\longrightarrow}\frac12
$$
(using for instance the mean value theorem in determining the limit), along with a similar result when $\varphi_0 \leq -\sqrt{2u}$ instead, and observing that $F^{+} (0)= F^{-} (0)$, cf. \eqref{E:7.16}, one readily infers that
\begin{equation}
\label{E:7.17}
\lim_{\varepsilon \searrow 0} \mathbf{P}_N^G[\psi_{\cdot}^u \in \cdot \, | \, \psi_0^u < \varepsilon] = \mathbf{P}_N^G\bigg[ \frac12 \Big( \Phi_{\cdot}\big(- \sqrt{2u}\, \big) + \sqrt{2u}\Big)^2 \in \cdot\bigg].
\end{equation}
On the other hand, using \eqref{E:7.14} and conditioning on $\varphi_0$, one also has
\begin{equation}
\begin{split}
\label{E:7.18}
\mathbf{P}_N^G[\psi_{\cdot}^u \in \cdot \, | \, \psi_0^u < \varepsilon]
&=\overline{P}_{x_*}^{N} \otimes \mathbf{P}_N^G\Big[ \Big(\bar\ell_{\sigma_u^N}^{N,\cdot} + \frac12 \Phi_{\cdot}(\varphi_0)^{\,2}\Big)\in \cdot \, \Big | \bar\ell_{\sigma_u^N}^{N,0} + \frac12 \varphi_{0}^{\,2}<  \varepsilon\Big].
\end{split}
\end{equation}
Hence, introducing, for $\delta > 0$,
\begin{equation}
\label{E:7.19}
G_{\delta}(h) = \overline{P}_{x_*}^{N}\Big[ \Big(\bar\ell_{\sigma_u^N}^{N,\cdot} + \frac12 \Phi_{\cdot}(h)^{\,2}\Big)\in \cdot, \, \bar\ell_{\sigma_u^N}^{N,0} \leq \delta \Big],  \quad \text{ for } h \in \mathbb{R}, \, \delta > 0,
\end{equation}
which is continuous in $h$ and increasing in $\delta$, one obtains using \eqref{E:7.18} and noting that $ \bar\ell_{\sigma_u^N}^{N,0}$ and $\frac12 \varphi_0^2$ are both non-negative, for all $\delta > 0$ and $\varepsilon < \delta$, 
\begin{equation}
\label{E:7.20}
\begin{split}
\mathbf{P}_N^G[\psi_{\cdot}^u \in \cdot \, | \, \psi_0^u < \varepsilon] \leq \frac{\mathbf{E}_N^G[G_{\varepsilon}(\varphi_0) 1\{\frac12 \varphi_0^2 < \varepsilon \}]}{ \overline{P}_{x_*}^{N}[\bar\ell_{\sigma_u^N}^{N,0} = 0] \mathbf{P}_N^G[\frac12 \varphi_0^2 < \varepsilon]} \leq \frac{\sup_{|h|< \sqrt{2\varepsilon}}G_{\delta}(h)}{\overline{P}_{x_*}^{N}[\bar\ell_{\sigma_u^N}^{N,0} = 0]},
\end{split}
\end{equation}
which readily yields, taking first $\delta \to 0$, then $\varepsilon \to 0$,
\begin{equation}
\label{E:7.21}
\limsup_{\varepsilon \to 0}\mathbf{P}_N^G[\psi_{\cdot}^u \in \cdot \, | \, \psi_0^u < \varepsilon] \leq \frac{G_{0}(0)}{\overline{P}_{x_*}^{N}[\bar\ell_{\sigma_u^N}^{N,0} = 0]}.
\end{equation}
Similarly, one has the lower bound
\begin{equation*}
\mathbf{P}_N^G[\psi_{\cdot}^u \in \cdot \, | \, \psi_0^u < \varepsilon] \geq \frac{\mathbf{E}_N^G[G_{0}(\varphi_0) 1\{\frac12 \varphi_0^2 < \varepsilon \}]}{ \overline{P}_{x_*}^{N}[\bar\ell_{\sigma_u^N}^{N,0} < \varepsilon] \mathbf{P}_N^G[\frac12 \varphi_0^2 < \varepsilon]} \geq \frac{\inf_{|h|< \sqrt{2\varepsilon}}G_{0}(h)}{ \overline{P}_{x_*}^{N}[\bar\ell_{\sigma_u^N}^{N,0} < \varepsilon]},
\end{equation*}
which, upon letting $\varepsilon \to 0$ and along with \eqref{E:7.21}, implies that
\begin{equation}
\label{E:7.22}
\lim_{\varepsilon \searrow 0} \mathbf{P}_N^G[\psi_{\cdot}^u \in \cdot \, | \, \psi_0^u < \varepsilon] = \frac{G_{0}(0)}{\overline{P}_{x_*}^{N}[\bar\ell_{\sigma_u^N}^{N,0} = 0]} \stackrel{\eqref{E:7.19}}{=} \overline{P}_{x_*}^{N}\Big[ \Big(\bar\ell_{\sigma_u^N}^{N,\cdot} + \frac12 \Phi_{\cdot}(0)^{\,2}\Big)\in \cdot \, \Big| \, \bar\ell_{\sigma_u^N}^{N,0} = 0 \Big]
\end{equation}
The claim \eqref{E:7.12} then follows from \eqref{E:7.18}, \eqref{E:7.22} and Lemma \ref{L:7.1}, since $\widetilde{\varphi}_{\cdot}=\Phi_{\cdot}(0)$, see \eqref{E:7.10}, and because
$$
\Phi_{\cdot}\big(- \sqrt{2u}\, \big) + \sqrt{2u} \stackrel{\eqref{E:7.10}}{=} \Phi_{\cdot}(0)- \sqrt{2u}P_{\cdot}[H_0< T_{B_N}] + \sqrt{2u} \stackrel{\eqref{E:7.13}}{=} \widetilde{\varphi}_{\cdot} + h_{\cdot}^N(u).
$$
\end{proof}

\begin{rmk}
\label{R:7.29}
1) The fact that we pin at $0$ (rather than some other value) plays a special role. Indeed, looking at \eqref{E:7.14}, we have crucially used that the left-hand side is a sum of two \textit{non-negative} fields, hence forcing their sum to be $0$ is tantamount to requiring that they vanish individually. If one chooses to fix $\psi_0^u$ to some other value (which can be done, see \eqref{E:7.10}), a non-trivial convolution remains.\\
\noindent 2) Although we are only concerned with the Markov chain $\overline{X}_{\cdot}^N$, the above proof can be applied without changes to the setting considered for instance in Ch. 2.4, p.52 of \cite{Sz12b}, thus yielding a pinned version of the generalized second Ray-Knight theorem for any (recurrent) random walk on a finite weighted graph. \\
\noindent 3) One can also pin on more general sets $K$, thereby obtaining the following theorem. We omit its proof, which follows the lines of that above, with the necessary modifications. Denote by $\mathbf{P}_{K'}^G$ the law of the GFF killed outside $K'$, i.e. as in \eqref{E:7.9} but with $g_{K'}$ in place of $g_{B_N}$, see \eqref{E:2.13}. For $K \subset K' \subset \subset \mathbb{Z}^2$, $x \in \mathbb{Z}^2$, define the field
\begin{equation}
\label{E:40.40}
\begin{split}
\Phi_x^K((h_y)_{y\in K}) &= \varphi_x + \sum_{y\in K}P_x[H_K < T_{K'}, X_{H_K}=y](h_y-\varphi_y) \\
&= \Phi_x^K(0,\dots, 0) + E_x[h_{X_{H_K}}  1\{H_K < T_{K'}\}].
\end{split}
\end{equation}
This corresponds to a choice of regular conditional distribution for $\mathbf{P}_{K'}^G$ given the values of the field in $K$.
\end{rmk}
\begin{Tbis} $(K \subset K' \subset \subset \mathbb{Z}^2, \, u > 0)$
\label{T:7.12'}
\begin{equation}
\label{E:7.12'}
\begin{split}
&\text{The law of }\Big(L_{x,u} + \frac12  \big(\Phi_x^K(0,\dots, 0)\big)^{\,2}\Big)_{x\in K'} \text{ under } \mathbb{P}^{K,K'} \otimes \mathbf{P}_{K'}^G,\\
&\text{is the same as the law of } \Big( \frac12 \big( \Phi_x^K(0,\dots, 0) + h_x^{K,K'}(u)\big)^2 \Big)_{x\in K'} \text{ under $\mathbf{P}_{K'}^G$},
\end{split}
\end{equation}
where
\begin{equation}
\label{E:7.13'}
h_x^{K,K'}(u) = P_x[H_K >T_{K'}] (2u)^{1/2}.
\end{equation}
\end{Tbis}

We now return to the setup of Theorem \ref{T:7.12}, and aim to investigate the limit as~$N \to \infty$. If one keeps $u$ fixed in \eqref{E:7.12}, the resulting limiting statement will be an obvious equality in law. However, as already hinted at in \eqref{E:6.21}, we can expect something interesting to happen if we boost $u$ suitably. Recall the pinned Gaussian free field $\varphi^{\, p}$ from \eqref{E:7.9}, as well as the (pinned) infinite volume interlacement process, whose law is denoted by $\mathbb{P}^{0,\mathbb{Z}^2}$, cf. above \eqref{C:6.17}, along with its corresponding field of local times $(L_{x,u})_{x\in \mathbb{Z}^2}$, for $u > 0$. 

\begin{thm} (Pinned isomorphism theorem, infinite volume) 
\label{T:7.29}
\medskip
\noindent
For all $\alpha > 0$,  
\begin{equation}
\label{E:7.30}
\begin{split}
&\text{the law of } \Big(L_{x,\alpha} + \frac12 (\varphi^{\,p}_x)^{\,2}\Big)_{x\in \mathbb{Z}^2}, \text{ under $\mathbb{P}^{0,\mathbb{Z}^2} \otimes \mathbf{P}^G$,}\\
&\text{is equal to the law of } \Big( \frac12 \big(  \varphi^{\,p}_x + \sqrt{2\alpha} a(x) \big)^2 \Big)_{x\in \mathbb{Z}^2} \text{ under $\mathbf{P}^G$}.
\end{split}
\end{equation}
\end{thm}

\begin{proof}
Since all the relevant quantities in \eqref{E:7.30} vanish when $x=0$, we may assume that $x \neq 0$. To begin with, we note that for all $N \geq 1$, $x \in B_N \setminus \{ 0\}$, since $P_0[\widetilde{H}_0 > T_{B_N}] = g_{B_N}(0,0)^{-1}$, which follows from \eqref{E:4.4} with $z=0$, $K=\{ 0\}$, 
\begin{equation}
\begin{split}
\label{E:7.31}
\mathbf{E}_N^G\Big[ \big(P_x[H_0>T_{B_N}]\varphi_0\big)^2\Big] &\stackrel{\eqref{E:7.9}}{=} P_x[H_0>T_{B_N}]^2 g_{B_N}(0,0)\\
&\stackrel{\eqref{E:4.4}}{=}\bigg[1 - \frac{g_{B_N}(x,0)}{g_{B_N}(0,0)}\bigg]^2 g_{B_N}(0,0) \stackrel{N \to \infty}{\longrightarrow} 0,
\end{split}
\end{equation}
using \eqref{E:4.2} and $g_{B_N}(0,0) \sim \frac2\pi \log N$ to compute the limit. By \eqref{E:7.10.0} and \eqref{E:7.10}, we have
\begin{equation*}
\widetilde{\varphi}_x = \varphi_x - P_x[H_0<T_{B_N}]\varphi_0 = \varphi_x - \varphi_0 +  P_x[H_0>T_{B_N}]\varphi_0,
\end{equation*}
hence, using \eqref{E:7.31} and Cauchy-Schwarz, and in view of \eqref{E:2.17}, it follows that $ \lim_N \mathbf{E}_N^G[\widetilde{\varphi}_x\widetilde{\varphi}_y]= \mathbf{E}^G[\varphi^{\,p}_x \varphi^{\,p}_y] $. Thus, by looking at characteristic functions, we obtain that
\begin{equation}
\label{E:7.32}
\widetilde{\varphi}_{\cdot} \text{ (under $\mathbf{P}_N^G$) converges in distribution towards } \varphi^{\,p}_{\cdot} \text{ (under $\mathbf{P}^G$),}
\end{equation} 
where convergence in distribution is meant in the sense of convergence of all finite-dimensional marginals. The claim \eqref{E:7.30} then follows from Theorem \ref{T:7.12} applied with $u=\alpha ( \frac{2}\pi \log N)^2$, by letting $N \to \infty$ and using \eqref{E:7.32}, Lemma \ref{L:6.20} and observing that
\begin{equation}
h_x^N\bigg(\alpha \Big( \frac{2}\pi \log N\Big)^2\bigg) \stackrel{\eqref{E:7.13}}{=} \sqrt{2\alpha} \, \bigg( \frac{2}\pi \log N P_x[H_0 > T_{B_N}] \bigg) \stackrel{N}{\longrightarrow} \sqrt{2\alpha} \, a(x),
\end{equation}
using \eqref{E:6.16.0} in the last step.
\end{proof}

As an immediate application of \eqref{E:7.30}, we note the following
\begin{corollary} $($Asymptotics for local times$)$

\medskip
\noindent One has the following limits in distribution regarding the field of local times $(L_{x,\alpha})_{x \in \mathbb{Z}^2}$ under $\mathbb{P}^{0,\mathbb{Z}^2}$: as $\alpha \to \infty$,
\begin{equation}
\label{E:7.40}
\Big(\frac{L_{x,\alpha}}{\alpha}\Big)_{x \in \mathbb{Z}^2} \to \big(a^2(x)\big)_{x \in \mathbb{Z}^2},\qquad \Big( \frac{L_{x,\alpha}-\alpha a^2(x)}{\sqrt{2\alpha}a(x)}\Big)_{x \in \mathbb{Z}^2} \to (\varphi_x^{\,p})_{x\in \mathbb{Z}^2}.
\end{equation}
\end{corollary}
\begin{proof}
The first item in \eqref{E:7.40} follows readily from \eqref{E:7.30}, noting that for every $x \in \mathbb{Z}^2$, $(\varphi_x^{\,p})^2 /\alpha \to 0$ and $( \varphi^{\,p}_x + \sqrt{2\alpha} a(x))^2 / 2\alpha \to a^2(x)$, $\mathbf{P}^G$-a.s. as $\alpha \to \infty$. The second claim follows similarly. 
\end{proof}

%%%%%%

%\begin{comment}
\section{%Construction by massive interlacements
 %Construction by 
 Limits of massive models}
\label{S:mass}

We now present a different approach to building the interlacements corresponding to \eqref{E:1.3}, which has the advantage of proceeding immediately in infinite volume, but uses a suitably tuned killing parameter $\epsilon$ for the random walks. Recalling from \eqref{E:10.0} that $\xi(\epsilon)$ is an exponential random variable of parameter $\epsilon$ under $P_x$, independent of the process $(X_t)_{t\geq 0}$, we define by $P_{\epsilon,x}$ the canonical law of 
\begin{equation}
\label{E:10.1}
Y_t = X_{t\wedge \xi(\epsilon)},  \quad t\geq0.
\end{equation}
 By adding a cemetery state $x_*$ not in $\mathbb{Z}^2$ as in \eqref{E:5.0} and redefining $Y_t = x_*$, $t \geq \xi$, $P_{\epsilon,x}$ is canonically viewed as a probability measure on $W$, cf. above \eqref{E:5.1}. Note in particular that $Y_{\cdot}$ is transient under $P_{\epsilon,x}$ for any $\epsilon > 0$. Its Green function is precisely $g_{\epsilon}(\cdot,\cdot)$, as defined in \eqref{eq:GF0}, i.e. 
\begin{equation}
\label{eq:GF}
\begin{split}
g_{\epsilon}(x,y) &= E_{\epsilon,x}\Big[\int_0^{\infty} \text d t \, 1\{ Y_t =y\} \Big], \qquad x, y \in \mathbb{Z}^2.
\end{split}
\end{equation}
A straightforward calculation shows that $
 g_{\epsilon}(x,y) = \sum_{n \geq 0} P_x[Z_n=y](1+\epsilon)^{-n}$
(recall that $Z_{\cdot}$ refers to the discrete skeleton of $X_{\cdot}$ under $P_x$). Hence, one may regard $Y_{\cdot}$ as a (continuous time, unit speed) Markov chain on the transient weighted graph $\mathbb{Z}^2 \cup\{ x_*\}$ endowed with the conductances $c_{x,y} =1/2d(1+\epsilon)$, for $|x-y|=1$, $x,y \in \mathbb{Z}^2$, $c_{x,x_*}= \epsilon/(1+\epsilon)$ and $c_{x_*,x_*}=1$. Hence, following \cite{Tei09}, we write 
\begin{equation}
\label{E:10.P}
\begin{split}
&\text{$\mathbb{P}_{\epsilon}$, $\epsilon > 0$, for the canonical law on $\Omega$, see \eqref{E:5.10}, of the Poisson}\\
&\text{point process with intensity measure $ \nu_{\epsilon}^*(\text dw^*) \text du$ (on $W^*\times [0,\infty))$,} 
\end{split}
\end{equation}
where $\nu_{\epsilon}^*$ is defined by $\nu_{\epsilon}^* \restriction{W_A^*} = Q_{\epsilon,A} \circ (\pi^*)^{-1}$, for $A\subset \subset \mathbb{Z}^d$, and $Q_{\epsilon,A}$ is a measure supported on the set of bi-infinite trajectories entering $A$ at time $0$ with
\begin{equation}
\label{E:10.2}
\begin{split}
&Q_{\epsilon,A} [(X_{-t})_{t\geq 0} \in E^-, X_0 = x, (X_t)_{t\geq 0} \in E^+] \\
&\qquad= P_{\epsilon,x} [(Y_t)_{t\geq 0} \in E^-  \,| \, \widetilde{H}_A = \infty ] \cdot e_{\epsilon,A}(x) \cdot P_{\epsilon,x} [(Y_t)_{t\geq 0} \in E^+  ],
\end{split}
\end{equation}
for $E^{\pm} \in \mathcal{W}_+$, with $Y_{\cdot}$ as in \eqref{E:10.1} and the corresponding equilibrium measure
\begin{equation}
\label{E:10.3}
e_{\epsilon,A}(x) = P_{\epsilon,x}[\widetilde{H}_A = \infty].
\end{equation}
We first determine the right scaling for $\epsilon$ in terms of $N$ as appearing in \eqref{E:1.3}, by giving a representation of the harmonic measure of $A$ in terms of the equilibrium measure $e_{\epsilon,A}$ in \eqref{E:10.3}. The correct choice is naturally governed by the relevant time scale $t_N$ given by \eqref{E:1.4.0}. One could also choose to eliminate $N$ altogether and scale $u$ suitably with the mass $\epsilon$ upon letting $\epsilon \to 0$ but this would obscure the link to the actual random walk in \eqref{E:1.3}.
\begin{proposition}
\label{P:10.4}
 For $A \subset \subset \mathbb{Z}^2$ containing the origin, and any sequence $(\epsilon_N)_N$  with $\epsilon_N \in (0,\infty)$ for all $N$, satisfying
 \begin{equation}
 \label{E:10.4}
\epsilon_N \sim t_N^{-1}, \quad \text{ as }N\to \infty,
\end{equation}
one has
 \begin{equation}
\label{E:10.5}
\sup_{x \in A} \bigg| \frac{2\log N}\pi  e_{\epsilon_N,A}(x)- \textnormal{hm}_A(x) \bigg| =o_A(1), \quad \text{ as }N\to \infty.
 \end{equation}
\end{proposition}

\begin{proof}
Consider a fixed sequence $(\epsilon_N)_N$ satisfying satisfying \eqref{E:10.4} and, for $\delta \in (0,1)$, recalling our convention regarding $\xi(\epsilon_N)$ below \eqref{E:10.0}, let 
\begin{equation}
\label{E:10.6}
G_N^{\delta} = \{ H_{\partial B_{N^{1-\delta}}} < \xi(\epsilon_N) <  H_{\partial B_{N^{1+\delta}}}\}.
\end{equation}
 We will show that $G_N^{\delta}$ happens with high probability, in that
\begin{equation}
\label{E:10.7}
P_x[ G_N^{\delta}] = 1- o_{A,\delta}\big( (\log N )^{-1} \big), \text{ as $N \to \infty$, for all $x \in A$}.
\end{equation}
Indeed, if \eqref{E:10.7} holds, then in view of \eqref{E:10.3}, noting that $P_{\epsilon,x}[\widetilde{H}_A = \infty] = P_{x}[\widetilde{H}_A > \xi(\epsilon)]$, cf. \eqref{E:10.1}, and by definition of $G_N^{\delta} $ in \eqref{E:10.6}, one infers, applying \eqref{E:10.7} twice, that for all $x \in A$,
\begin{equation}
\label{E:10.8}
\begin{split}
&\frac{2\log N}\pi P_{x}[\widetilde{H}_A > H_{\partial B_{N^{1+\delta}}}] + o_{A,\delta}(1) \\
&\qquad \qquad \leq \frac{2\log N}\pi  e_{\epsilon_N,A}(x) \leq \frac{2\log N}\pi P_{x}[\widetilde{H}_A > H_{\partial B_{N^{1-\delta}}}] + o_{A,\delta}(1).
\end{split}
\end{equation}
Taking $N \to \infty$ for fixed $\delta>0$ in \eqref{E:10.8} and recalling \eqref{E:2.11} gives
$$
\frac{ \textnormal{hm}_A(x)}{1+ \delta} \leq \liminf_N\frac{2\log N}\pi  e_{\epsilon_N,A}(x) \leq  \limsup_N\frac{2\log N}\pi  e_{\epsilon_N,A}(x) \leq \frac{ \textnormal{hm}_A(x)}{1- \delta}, 
$$
for all $x\in A$, from which \eqref{E:10.5} follows upon letting $\delta \searrow 0$.

We now show \eqref{E:10.7}, and to this end, first note that, for all $x \in A$,
\begin{equation}
\label{E:10.9}
P_x[ H_{\partial B_{N^{1-\delta}}} \geq \xi(\epsilon_N)] \leq P_x[ H_{\partial B_{N^{1-\delta}}} \geq N^2 ] + P_x[  \xi(\epsilon_N) <N^2].
\end{equation}
Using the (crude) estimate $E_x[H_{ B_{N^{1-\delta}}^c}] \leq cN^{2(1-\delta)}$ valid for all $x \in B_{N^{1-\delta}}$ and $\delta \geq 0$, see \eqref{E:2.exp3}, a first moment bound yields that $P_x[ H_{\partial B_{N^{1-\delta}}} \geq N^2 ] \leq c N^{-2\delta}$. Using the elementary inequality $e^{-x} \geq 1-x$, for $x> 0$, we also have that $P_x[  \xi(\epsilon_N) <N^2] = 1 -e^{-\epsilon_N N^2}\leq \epsilon_N N^2$, hence
\begin{equation}
\label{E:10.10}
P_x[ H_{\partial B_{N^{1-\delta}}} \geq  \xi(\epsilon_N)] \leq c N^{-2\delta} \wedge \epsilon_N N^2 = O_{A,\delta} \big( (\log N )^{-2} \big),
\end{equation}
on account of \eqref{E:10.4} and \eqref{E:1.4.0}. We now derive a suitable upper bound for $P_x[ H_{\partial B_{N^{1+\delta}}} \leq  \xi(\epsilon_N)]$, $x \in A$. %The following argument %is a coarse-grained local CLT-type argument for the random walk that 
%combines well with the memoryless property of $\xi(\epsilon_N)$. 
For a parameter $R \geq 1$ to be chosen soon, define the successive stopping times
\begin{equation}
\label{E:10.12}
H_1= H_{\partial B_R(X_0)}, \quad H_{k}= H_{k-1} + H_1 \circ \theta_{H_{k-1}}, \, k \geq 2.
\end{equation}
For convenience, let $\widetilde{\xi}\stackrel{\text{law}}{=}\xi (=\xi(\epsilon_N))$, be a copy of $\xi$  under $\widetilde{P}$ independent of $\xi$ and $\{ X_t: t \geq 0 \}$. Then, with $\mathcal{F}_{H_{k-1}}$ denoting the $\sigma$-algebra of the past of $H_{k-1}$, using first the independence of $\xi$ and $\{ X_t: t \geq 0 \}$, then the memoryless property of the exponential, and the strong Markov property at time $H_{k-1}$, we find, for any $x\in \mathbb{Z}^2$, $P_x$-a.s. (noting that $ H_{k-1} < \infty$, $P_x$-a.s.) 
\begin{equation}
\label{E:10.13}
\begin{split}
&E_x\big[ 1\{ \xi \geq H_{k} \} \, \big| \, \mathcal{F}_{H_{k-1}} \big]  = E_x\big[ \, \widetilde{P}[\widetilde \xi \geq H_{k-1} + H_1 \circ \theta_{H_{k-1}} ] \, \big| \mathcal{F}_{H_{k-1}} \big]\\
&\qquad = E_x\big[ \, \widetilde{P}[\widetilde \xi \geq  H_1 \circ \theta_{H_{k-1}} ]\cdot \widetilde{P}[\widetilde \xi \geq H_{k-1} ] \, \big| \mathcal{F}_{H_{k-1}} \big] \\
&\qquad = E_x\big[ 1\{ \xi \geq H_1 \circ \theta_{H_{k-1}} \} \, \big| \, \mathcal{F}_{H_{k-1}} \big] \cdot \widetilde{P}[\widetilde \xi \geq H_{k-1} ]\\
&\qquad = P_{X_{H_{k-1}}}[\xi \geq H_1]\cdot \widetilde{P}[\widetilde \xi \geq H_{k-1} ]
\end{split}
\end{equation}
where we also used that $\widetilde{P}[\widetilde \xi \geq H_{k-1} ]$ is $\mathcal{F}_{H_{k-1}}$-measurable. Applying \eqref{E:10.13} inductively, and since $ x\mapsto P_{x}[\xi \geq H_1]$ is stationary under spatial shifts, cf. \eqref{E:10.12}, it follows that 
\begin{equation}
\label{E:10.14}
P_x[  \xi \geq H_{k} ] = E_x\big[ E_x\big[ 1\{ \xi \geq H_{k} \} \, \big| \, \mathcal{F}_{H_{k-1}} \big] \big] = P_{0}[\xi \geq H_1]^k, \text{ for all $k \geq 1$, $x\in \mathbb{Z}^2$}.
\end{equation}
(Alternatively, one can also deduce \eqref{E:10.14} by considering $Y_{\cdot}$ defined in \eqref{E:10.1} on an extended graph, cf. the discussion below \eqref{eq:GF}, thus identifying $\xi=H_{\{x_*\}}$ and using the strong Markov property). Now, integrating over $ \xi =\xi(\epsilon_N)$
on the right-hand side of \eqref{E:10.14} and using \eqref{E:2.exp1} yields
\begin{equation}
\label{E:10.15}
P_x[  \xi(\epsilon_N) \geq H_{k} ] = E_0\big[  e^{-\epsilon_N H_1}\big]^k \leq (1-c\epsilon_NR^2)^k, \quad  \text{ for $k \geq 1$, $x\in \mathbb{Z}^2$, and $N \geq c'(R)$},
\end{equation}
such that $\epsilon_N \leq \epsilon_0(R)$ for all $N \geq c'(R)$ with $\epsilon_0$ as given in Lemma \ref{L:2.exp} (note that $\epsilon_N \to 0$ as $N \to \infty$ by \eqref{E:10.4}). 

The typical number $k$ of $R$-boxes crossed before exiting $B_{N^{1+\delta}}$ scales diffusively, cf. \eqref{E:10.12}. This yields the following

\begin{lem} $( A \subset \subset \mathbb{Z}^2, \, R\geq 1, \delta\in (0,\frac1{10}))$
\label{L:10.16}

\medskip
\noindent
Let $\kappa_N^\delta= \sup\{ k \geq 1: \, H_k < H_{\partial B_N^{1+ \delta}}  \}$. Then, for suitable $c,c'\in (0,\infty)$ depending on $A, R, \delta$ only, and all $N \geq 1$,
\begin{equation}
\label{E:10.16}
 P_x[\kappa_N^\delta < t_N \log N ] \leq ce^{-N^{c'}}, \quad x \in A.
\end{equation}
\end{lem}
We defer the proof of Lemma \ref{L:10.16} for a few lines. Assuming \eqref{E:10.16} to hold, and with $k_N = t_N \log N $, we thus obtain, for $x \in A$,
\begin{equation}
\label{E:10.17}
\begin{split}
&P_x[ H_{\partial B_{N^{1+\delta}}} \leq  \xi(\epsilon_N)] \leq P_x[ H_{\kappa_N^\delta} \leq  \xi(\epsilon_N)] \\
&\qquad \stackrel{\eqref{E:10.12}}{\leq} P_x[ H_{k_N} \leq  \xi(\epsilon_N)] + P_x[\kappa_N^\delta < t_N \log N ]\\
&\qquad \stackrel{\eqref{E:10.15}}{\leq} \big(1-c't_N^{-1}R^2\big)^{k_N} + ce^{-N^{c'}},
\end{split}
\end{equation}
where we also used \eqref{E:10.4}. Thus, choosing for instance $R=100$, \eqref{E:10.17} implies that $P_x[ H_{\partial B_{N^{1+\delta}}}  \leq  \xi(\epsilon_N)]$ decays (at least) polynomially in $N$ (with constants depending on $A$ and $\delta$). Together with \eqref{E:10.10}, this yields \eqref{E:10.7}, cf. also \eqref{E:10.6}, and thus completes the proof of \eqref{E:10.5}.
\end{proof}

It remains to give the

\begin{proof10}
By projecting onto each of the coordinates of $X_{\cdot}$ and observing that, under $P_0$, in order to exit $B_{N^{1+\delta}}$, the random walk must exit at least $\lfloor N^{1+\delta}/R \rfloor$ boxes among $\{ B_{R}(X_{H_k}), \, k \geq 1\}$ either horizontally in the same direction (i.e. all through the left or all through the right), or vertically in the same direction, and using the strong Markov property, one finds that
\begin{equation}
\label{E:10.18}
P_x[\kappa_N^\delta < t_N \log N ] \leq 2 \overline{P} \bigg[\, \bigg| \sum_{1 \leq k \leq  t_N \log N} \overline{Z}_k\, \bigg| \geq  \lfloor N^{1+\delta}/R \rfloor\bigg],
\end{equation}
where $\overline{Z}_k$, $k \geq 1$ are independent and identically distributed under $\overline{P}$, and equal to $\pm 1$ with probability $\frac14$ each, and otherwise equal to $0$. The claim \eqref{E:10.16} then quickly follows from \eqref{E:10.18} standard concentration estimates for sums of (bounded) independent random variables, see for instance \cite{LL10}, Corollary 12.2.7.\hfill $\square$
\end{proof10}

%Roughly speaking, as the walk crosses successive $R$-boxes, the probability for it not to have been killed before completing the $k$-th crossing comes at a certain (small) cost, as given by Lemma ..., independently for successive crossings, as follows from the Markov property and the memorylessness of $\xi$, The total number of $R$-boxes crossed before reaching $H_{\partial B_{N^{1+\delta}}}$ is concentrated around $(N/R)^{2(1+ \delta)}$, which is sufficiently large to make $P_x[ H_{\partial B_{N^{1+\delta}}} \leq  \xi(\epsilon_N)] \approx (1-c\epsilon_N R^2)^{(N/R)^{2(1+ \delta)}}$ small for suitable choice of $R$.

%The Green function \eqref{eq:GF} for the massive walk, under the scaling \eqref{E:10.4}, admits the following large $N$ asymptotics.

We return to the Poisson random measure $\omega$ with law $\mathbb P_{\epsilon}$ defined in \eqref{E:10.P}, and are now ready to state the main result of this section. Given $u>0$, we write $\mathbb P_{\epsilon,u}$ for the law of the random measure $\omega_u$ obtained by collecting all trajectories in $\omega$ with label at most $u$ - its intensity measure (on $W^*$) is $u\, \nu_{\epsilon}(\text d w^*)$. As in \eqref{E:5.13}, we write $\mathcal{I} = \mathcal{I}(\omega_u) \subset \mathbb{Z}^2$ for the random set consisting of all sites which are visited by at least one of the trajectories in the support of $\omega_u$. Recall the measure $Q^{\alpha}$, $\alpha >0$ from \eqref{E:1.5}. In referring to the law of $\mathcal{I}$ below, we mean the law of $(1\{ x\in \mathcal{I} \})_{x}$ (on $(\widetilde{\Omega},\widetilde{\mathcal F})$).

\begin{thm} $(\alpha > 0)$
\label{T:10.40}

\medskip
\noindent For all $(\epsilon_N)_N$ as in \eqref{E:10.4} and $u_N = u_N(\alpha)$ satisfying \eqref{E:6.2},
\begin{equation}
\label{E:10.40}
\text{the law of $\mathcal{I}$ under $\mathbb{P}_{\epsilon_N,u_N}[\, \cdot \, | 0\notin \mathcal{I}]$ converges in distribution to $Q^{\alpha}$}.
\end{equation}
\end{thm}
\begin{proof}
For $A\subset \subset \mathbb{Z}^2$ containing the origin, let
\begin{equation}
\label{E:10.41}
\Xi_N (A) = \mathbb{P}_{\epsilon_N,u_N} [\mathcal{I} \cap A =\emptyset \, |\, 0\notin \mathcal{I}].
\end{equation}
We have
\begin{equation}
\label{E:10.42}
\begin{split}
&\Xi_N (A) = \exp\big[ - u_N (\nu_{\epsilon_N}^* (W_A^*) - \nu_{\epsilon_N}^* (W_0^*) ) \big] \\
&\quad\stackrel{\eqref{E:10.2}}{=} \exp \bigg[ - u_N \Big( \sum_{x\in A} e_{\epsilon_N,A}(x) - e_{\epsilon_N,\{0\}}(0) \Big)\bigg].
\end{split}
\end{equation}
By last exit-decomposition for the killed walk $Y_{\cdot}$, we know that for all $z \in \mathbb{Z}^2$ and finite $K \subset \mathbb{Z}^2$,
\begin{equation}
\label{E:10.43}
P_{\epsilon_N,z} [H_K < \infty]= P_{z} [H_K <  \xi(\epsilon_N)] = \sum _{y\in K} g_{\epsilon_N} (z,y) e_{\epsilon_N,K}(y).
\end{equation}
Applying \eqref{E:10.43} with $K=A$ and $\{ 0\}$ in \eqref{E:10.42}, we readily obtain for all $z$,
\begin{equation}
\label{E:10.44}
\begin{split}
-\log \Xi_N (A) &= \frac{u_N}{g_{\epsilon_N}(z,0)} \big( P_{z} [H_A <  \xi(\epsilon_N)] - P_{z} [H_0 <  \xi(\epsilon_N)] \big)\\
&\quad + \frac{u_N}{g_{\epsilon_N}(z,0)} \sum_{y\in A} \Big(g_{\epsilon_N}(z,0) - g_{\epsilon_N}(z,y) \Big)  e_{\epsilon_N,A}(y).
\end{split}
\end{equation}
Choosing $z=0$, the first term on the right-hand side of \eqref{E:10.44} vanishes. Moreover, as $N \to \infty$,
\begin{equation}
\label{E:10.45}
g_{\epsilon_N}(0,0) \stackrel{\eqref{E:10.43}}{=}e_{\epsilon_N,\{ 0 \}}(0)^{-1} \stackrel{\eqref{E:10.5}}{=} \frac{2}{\pi} \frac{\log N}{1+o(1)},  
\end{equation}
hence $u_N/ g_{\epsilon_N}(0,0) \sim \alpha \log N$, on account of \eqref{E:6.2}. Finally, recalling \eqref{E:2.12} (and choosing $y=0$), approximating $\text{hm}_A$ by $e_{\epsilon_N, A}$ as in \eqref{E:10.5} and using \eqref{E:10.20}, we can write 
\begin{equation}
\label{E:10.31}
\textnormal{cap}(A)=  \frac2\pi \lim_{N \to \infty} \log N \sum_{x\in A} \big(g_{\epsilon_N}(0,0) - g_{\epsilon_N}(0,x) \big)\, e_{\epsilon_N, A}(x).
\end{equation}
Thus, returning to \eqref{E:10.44}, we obtain that
\begin{equation}
\label{E:10.45}
-\log \Xi_N (A)  \stackrel{N \to \infty}{\longrightarrow} \frac{\pi}{2} \,\alpha\, \text{cap}(A).
\end{equation}
In view of \eqref{E:10.41} and \eqref{E:1.5}, this concludes the proof.
\end{proof}

\begin{rmk}\label{R:mass}
1) Behind the proof of Theorem \ref{T:10.40} lurks a formula much in the spirit of \eqref{E:4.1}. Namely, for all $ A \subset\subset \mathbb{Z}^2$, $y \in A$,  and $( \epsilon_N )_N$ as in \eqref{E:10.4}, 
\begin{equation}
\label{E:10.30}
\textnormal{cap}(A)=  \lim_{N \to \infty} \bigg( \frac{2}{\pi} \log N \bigg)^2 \sum_{x\in A} e_{\epsilon_N, A}(x) \,  P_{\epsilon_N,x}[H_y =\infty].
\end{equation}
Indeed, following the steps of the proof of Lemma \ref{L:4.cap}, starting with \eqref{E:10.31} (with $y \, (\in A)$ instead of $0$) in place of \eqref{E:4.3}, and noting that $e_{\epsilon, A}$ satisfies a sweeping identity as \eqref{E:2.15}, \eqref{E:10.30} readily follows.\\
2) Theorem \ref{T:10.40} can be strengthened, essentially by suitably adapting the proofs of Theorem \ref{T:LIMIT} and Corollary \ref{C:LIMIT}, to yield the convergence as $N \to \infty$ of the pinned process $\mathbb{P}_{\epsilon_N, u_N}[\, \cdot \, | 0\notin \mathcal{I}]$ of massive interlacements towards $\mathbb{P}^{0,\mathbb{Z}^2}$, cf. above \eqref{C:6.17} (with $\epsilon_N$, $u_N$ as above \eqref{E:10.40}). We will not discuss this in more detail here.\\
3) One can also use the measures $\mathbb{P}_{\epsilon_N, u_N}$ to recover Theorem~\ref{T:7.12}. We sketch how this can be done. One first introduces the field of local times $(L_{x}^{\epsilon,u})_{x\in \mathbb{Z}^2}$ attached to $\mathbb{P}_{\epsilon, u}$, similarly to \eqref{E:5.12}. Due to \cite{Sz12c}, Theorem 0.1, and the discussion following \eqref{eq:GF}, one knows that
\begin{equation}
\label{E:10.46}
L_{\cdot}^{\epsilon,u} + \frac12 (\varphi^{\epsilon}_{\cdot})^2 \stackrel{\text{law}}{=} \frac12 ( \varphi^{\epsilon}_{\cdot} +\sqrt{2u})^2, \quad \text{for all }\epsilon,u > 0, 
\end{equation}
where $\varphi^{\epsilon}_{\cdot}$ is the centered Gaussian field with covariance $g_{\epsilon}(\cdot,\cdot)$, cf. \eqref{E:massGFF}, under $\mathbf\mathbf{P}^G_{\epsilon}$, independent of $\mathbb{P}_{\epsilon, u}$ on the left-hand side. One then ``pins down''  \eqref{E:10.46} much in the same way as in the proof of Theorem \ref{T:7.12}, to obtain that
\begin{equation}
\label{E:10.47}
L_{\cdot}^{\epsilon,u} + \frac12 (\widetilde{\varphi}^{\epsilon}_{\cdot})^2 \stackrel{\text{law}}{=} \frac12 ( \widetilde{\varphi}^{\epsilon}_{\cdot} +h_{\cdot}^{\epsilon}(u))^2, \quad \text{for all }\epsilon,u > 0, 
\end{equation}
where $L_{\cdot}^{\epsilon,u}$ is now sampled according to the pinned measure $\mathbb{P}_{\epsilon, u}[\, \cdot \, | 0\notin \mathcal{I}]$, $h_{\cdot}^{\epsilon}(u)=P_{\cdot}[H_0 < \xi(\epsilon)] \sqrt{2u}$ and $\widetilde{\varphi}^{\epsilon}_{\cdot}$ is distributed according to $\mathbf\mathbf{P}^G_{\epsilon}[\, \cdot \, | \, \varphi^{\epsilon}_0 =0]$ (to make the latter precise, one proceeds as in \eqref{E:7.10}, \eqref{E:7.10.0}). Now, one follows the proof of Theorem \ref{T:7.29}, with \eqref{E:10.47} in place of \eqref{E:7.12}, and shows that 
\begin{equation*}
\begin{array}{rcl}
L_{\cdot}^{\epsilon_N,u_N} \text{ (under $\mathbb{P}_{\epsilon_N, u_N}[\, \cdot \, | 0\notin \mathcal{I}]$) } &\stackrel{d}{\longrightarrow} & L_{\cdot,\alpha}  \text{ (see \eqref{E:7.30})}, \\
\widetilde{\varphi}^{\epsilon_N}_{\cdot} &\stackrel{d}{\longrightarrow}& \varphi_{\cdot}^{\, p}  \text{ (see \eqref{E:2.16})}
\end{array}
\end{equation*}
(all convergences are meant in the sense of finite-dimensional marginals), and for the second line, which holds in fact true for any sequence $\epsilon_N$ going to $0$, one uses \eqref{E:12.1}, along with a calculation similar to \eqref{E:7.31}, and Lemma \ref{L:10.20}. \hfill $\square$
\end{rmk}

Our results invite a few concluding comments.

\begin{rmk}
\label{R:final}
The above Poissonian description (which at this point is a representation \textit{in law}) naturally seems to arise as a print of the actual random walk, and one can try to make this precise by coupling the two objects. One is also left to wonder what happens if the order of magnitude of the relevant time scale $t_N$ in \eqref{E:1.4}, \eqref{E:1.4.0} is suitably altered. We hope to return to these questions in future work.\hfill $\square$

\end{rmk}

\bibliography{rodriguez}

\bibliographystyle{plain}

%\begin{thebibliography}{100}
%{%\small
%\footnotesize
%\bibitem{A14}
%S.~Andres, \textit{Invariance principle for the random conductance model with dynamic bounded conductances}, 
%Ann. Inst. Henri Poincar\'e Probab. Stat. \textbf{50} (2014), no.~2, 352--374.

%}
%\end{thebibliography}

\end{document}